\documentclass[11pt]{article}
\usepackage{extsizes}
\usepackage[all]{xy}
\usepackage{mymacros2}

\newcommand{\wB}{\widehat{B}}
\newcommand{\wQ}{\widehat{Q}}
\newcommand{\whC}{\widehat{C}}

\usepackage[margin=3cm]{geometry}

\usepackage{hyperref}

\newcommand{\Graph}{\mathrm{Graph}}

\newcommand{\Sol}{\mathrm{Sol}}
\newcommand{\SQ}{\mathrm{SQ}}

\newcommand{\Char}{\mathrm{Char}}
\newcommand{\musupp}{{\mu\supp}}
\newcommand{\potimes}{\overset{+}{\otimes}}
\newcommand{\Int}{\mathrm{Int}}
\newcommand{\hol}{\mathrm{hol}}
\newcommand{\forlim}{{\text{``${\displaystyle{\lim}}$''}}}

\newcommand{\wC}{\widetilde{\bC}}

\newcommand{\Open}{\mathrm{Open}}

\newcommand{\DR}{\mathrm{DR}}

\renewcommand{\Sh}{\mathrm{Sh}}
\renewcommand{\Mod}{\mathrm{Mod}}

\renewcommand{\SS}{\mathrm{SS}}

\newcommand{\al}{\mathrm{al}}

\newcommand{\cOL}{\cO^{\hbar, \Lambda_0}}
\newcommand{\cOae}{\cO^{\ae}}
\newcommand{\cOc}{\cO^{\mathrm{ce}}}
\newcommand{\cOse}{\cO^{\mathrm{se}}}

\newcommand{\cOexp}{\cO^{\exp}}
\newcommand{\cOaeL}{\cO^{\ae, \Lambda_0}}
\newcommand{\cOseL}{\cO^{\mathrm{se},\Lambda_0}}
\newcommand{\cOdeL}{\cO^{\mathrm{de},\Lambda_0}}
\newcommand{\cOexpL}{\cO^{\Lambda}}
\newcommand{\cOceL}{\cO^{\mathrm{ce}, \Lambda_0}}
\newcommand{\se}{\mathrm{se}}
\newcommand{\ce}{\mathrm{ce}}
\newcommand{\cDae}{\cD^{\ae}}

\newcommand{\cDaeL}{\cD^{\ae, \Lambda_0}}
\newcommand{\cDseL}{\cD^{\mathrm{se},\Lambda_0}}

\newcommand{\cDceL}{\cD^{\mathrm{ce}, \Lambda_0}}
\newcommand{\cDdeL}{\cD^{\mathrm{de}, \Lambda_0}}
\newcommand{\noneq}{\mathrm{ne}}

\title{$\hbar$-Riemann--Hilbert correspondence}
\author{Tatsuki Kuwagaki}
\date{}

\begin{document}

\maketitle
\begin{abstract}
We formulate and prove a Riemann--Hilbert correspondence between $\hbar$-differential equations and sheaf quantizations, which can be considered as a correspondence between two kinds of quantizations (deformation and sheaf quantization) of holomorphic cotangent bundles. The latter category is expected to be equivalent to a version of Fukaya category, which is a ``quantization" of Lagrangian intersection theory. The ideas of the constructions are based on asymptotic/WKB analysis, which is related to geometric quantization.
\end{abstract}

\setcounter{tocdepth}{1}
\tableofcontents
\section{Introduction}
\subsection{Sato's lost dimension}
Solving complex linear differential equations has been a central topic in algebraic analysis. Mikio Sato initiated $\cD$-module theory and also emphasized the importance of seeing microlocal directions. 

One of the achievements of such philosophy was realized by the Riemann--Hilbert correspondence formulated and proved by Kashiwara~\cite{KashiwaraRH} (another proof by Mebkhout~\cite{Mebkhout}). Microlocal theories exist in this story on both sides: $\cE$-module theory~\cite{SKK} and microlocal sheaf theory~\cite{KS}. 

Later years, Sato's Kyoto school started to study exact WKB analysis~\cite{KawaiTakei}. Exact WKB analysis is a special way to solve differential equations by the Borel resummation method applied to WKB
solutions~\cite{Voros}. Sato imagined this theory should recover the lost dimension in microlocal theories~\cite{SatoMath}. Namely, in the microlocal theories, objects are always conic with respect to the fiber scaling of cotangent bundles i.e., there is no information in this direction outside the zero section.

The aim of this paper is to give a first step to realize Sato's imagination, in the context of Riemann--Hilbert correspondence (by developing the author's previous work~\cite{kuwagaki2020sheaf}). Namely, we would like to provide a theory which is a mixture of $\cD$-module theory, microlocal sheaf theory, and (exact) WKB analysis.

\subsection{Quantization and Fukaya category}
One can also motivate our theory developed here from a symplectic geometric point of view. 

There are several formalisms of quantization of symplectic geometry. Deformation quantization is one of the ways. A deformation quantization of a symplectic manifold is a deformation of the structure sheaf algebra of the manifold over $\hbar$ whose 1-st order correction is given by the Poisson bracket. One of the simplest instance of deformation quantization is the ring of differential operators on a manifold, which can be considered as a deformation quantization of the cotangent bundle with $\hbar=1$. 

On the other hand, we have the notion of Fukaya category. Roughly speaking, this is a category associated to a symplectic manifold whose objects are Lagrangian submanifolds (+ brane data) and the space of morphisms is given by intersection Floer theory. We can consider that Fukaya category is also a kind of quantization in the sense that it deforms the classical intersection theory by instanton corrections.

These two seemingly unrelated quantizations are expected to be related~\cite{BresslerSoibelman} , and actually related in the case of cotangent bundles of complex manifolds:
\begin{equation}
    D^b_{rh}(\cD_M)\hookrightarrow \Sh^c(M)\cong \Fuk(T^*M).
\end{equation}
Here $D^b_{rh}(\cD_M)$ is (derived) the category of regular holonomic modules over a complex manifold $M$, $\Sh^c(M)$ is the category of real constructible sheaves (complexes), and $\Fuk(T^*M)$ is Fukaya category of $T^*M$ (with an appropriate wrapping setup). The first arrow is the regular Riemann--Hilbert correspondence and the second arrow is the Nadler--Zaslow equivalence~\cite{NZ} (or Ganatra--Pardon--Shende equivalence with an appropriate interpretation on the data of stops~\cite{GPS}).

In the latter equivalence, the correspondence is obtained (roughly speaking) by sending a sheaf to its microsupport. The microsupport of constructible sheaves are always conic Lagrangians. Hence the Lagrangians appearing in the above equivalence are potentially conic (i.e., conic at the contact boundary). This is the description of Sato's lost dimension in symplectic geometric setup. In other words, when studying symplectic geometry using sheaves naively, we can only study conic Lagrangians.

Tamarkin, in his seminal paper~\cite{Tam} (first appeared in 2008 on arXiv), proposed a way to remedy the situation. For an exact Lagrangian submanifold $L\subset T^*M$, using a primitive of the Liouville form restricted to $L$, he defined a lift $\widehat{L}\subset T^*M\times T^*\bR$. Then Tamarkin proved non-displaceability results using sheaves microsupported on $\widehat{L}$. Such a sheaf is called a sheaf quantization of $L$. From the microlocal-sheaf-theoretic point of view, the notion of sheaf quantization is a non-conic version of the notion of constructible sheaf. The idea of sheaf quantization is later developed by many people. Some of the main references are \cite{GKS, Guillermou, JT, AI, kuwagaki2020sheaf, nadlershende}. 

Tamarkin's idea is probably based on deformation quantization of Lagrangian submanifolds. Namely, the added real line $\bR$ can be considered as a real counterpart of (the Fourier--Laplace dual) of $\hbar$. Actually, such a variable has been effectively used to construct deformation quantization modules supported on (not necessarily conic) Lagrangian submanifolds~\cite{PS} (see also an explanation in \cite{kuwagaki2020sheaf}).  

The ring of differential operator ($\subset$ the conic world) is obtained by setting $\hbar=1$ for the deformation quantization ($\subset$ the non-conic world) of the cotangent bundle. Hence the usual Riemann--Hilbert correspondence can be considered as a comparison in the conic world. Then it is natural to expect a non-conic version of the Riemann--Hilbert correspondence, between deformation quantization and sheaf quantization.

For the 1-dimensional case, it is actually done by using exact WKB analysis in our previous paper on the object level~\cite{kuwagaki2020sheaf}. Namely, given an $\hbar$-differential equation with certain conditions on a Riemann surface, we construct a solution sheaf quantization using exact WKB analysis.

Our aim in this paper is to functorize the construction of \cite{kuwagaki2020sheaf} and state it in the form of an equivalence of categories. To achieve this goal, one can naively expect a realization of the Borel resummation procedure as a functor. However, it is not the case. Actually, exact WKB analysis breaks down over turning points, but sheaves should also exist over such points. In \cite{kuwagaki2020sheaf}, it is remedied by hand, but it is not functorial. Also, the availability of exact WKB analysis in higher orders or higher dimensions are still conjectural.

To avoid such difficulties, we use a weaker notion for solutions, which do not need to have asymptotic expansions. So, exact WKB analysis is not explicit in our main constructions, but it is related to our previous solution sheaf quantizations at the object level. This point will be explained in Section 16.

Before describing our results, we would like to mention Kontsevich--Soibelman's DQ-Riemann--Hilbert conjecture (see, for example, \cite{KonTalk}). The conjecture roughly says a relation between deformation quantization modules and Fukaya category for a certain class of holomorphic symplectic manifolds. We expect that there exists a Nadler--Zaslow (or GPS)-type equivalence for our sheaf quantizations over the Novikov ring, and it relates our theorem and their conjecture for the case of cotangent bundles. Completing this picture, we can obtain the non-conic version of the diagram (1.1).

\begin{remark}[Geoemtric quantization]
We can also motivate from another form of quantization: geometric quantization. The additional variable $\bR$ can be considered as a kind of the trivial prequantum bundle. Using the standard polarization, one can consider a section of a prequantum bundle as a state (see e.g. \cite{GuilleminSternberg, BatesWeinstein}). In the cotangent bundle setting, a Lagrangian submanifold ($+$ some data) gives a state, so-called WKB-state. In this terminology, our sheaf quantization can be considered as a kind of geometric (or topological?) avatar of a WKB-state associated to a Lagrangian submanifold obtained as a solution of the given differential equation, which is a deformation quantization of the Lagrangian submanifold.

Geometric, deformation, and Fukaya quantization interaction is recently taken up in the work of Gaiotto--Witten~\cite{gaiotto2021probing}, developing the ideas in \cite{kapustin2007electricmagnetic, Gukov}. It should be interesting to compare these works with our present work. Some observations will be treated in \cite{KuwaCC}.
\end{remark}

\begin{remark}
An approach to Fukaya category from a WKB-like consideration also appeared in \cite{Tsygan}.
\end{remark}

\begin{remark}
A version of the Riemann--Hilbert correspondence with a presence of $\hbar$ is also treated in \cite{DGS}. In one sentence, the difference is: Their treatment is formal, while our treatment is analytic. It might be interesting to relate our result with their result through the asymptotic expansion map.
\end{remark}

\begin{remark}[Algebraic Lagrangians]
In this paper, we only treat the class of algebraic Lagrangians (see Definition in Appendix). In particular, if $L$ is an algebraic Lagrangian in $T^*M$, the intersection $L\cap T^*_xM$ has only finite connected components. It should also be important to extend our results to non-algebraic Lagrangians for the following reason. If one wants to treat Kontsevich--Soibelmans's DQ-RH correspondence for the difference (or q-difference/ elliptic difference) equations sheaf-theoretically, a standard strategy is to consider them as infinite-order differential equations. Then the corresponding Lagrangians are  infinite-order coverings and not algebraic.
\end{remark}

\subsection{Main players}
As we have already mentioned, there exist several developments (i.e., generalizations) of the notion of sheaf quantization. In this paper, we use a slight generalization of the version in \cite{kuwagaki2020sheaf}. One advantage of the definition of \cite{kuwagaki2020sheaf} was that the category of sheaf quantization is enriched over the Novikov ring. The relationship between the Novikov ring and the extra $\bR$ was already explained in \cite{Tam}, but we have realized it more naturally using equivariant sheaves. Since we would like to consider $\hbar$ as a complex variable, we will use $\bC$ instead of $\bR$.

Since the topological side (i.e., sheaf quantization), is enriched over the Novikov ring, the differential equation side should also be enriched over the Novikov ring. We obtain such a category by completing the rings related to asymptotic analysis.

Let $\bC_\hbar$ be the complex line with the standard coordinate $\hbar$. Take the oriented real blow-up of the origin and take a neighborhood of an acute portion of the exceptional divisor. Such a subset (+ an extra condition) will be defined in the body of the paper and called a sectoroid.

Let $S$ be a sectoroid and $M$ be a complex manifold. We denote the sheaf of asymptotically expandable holomorphic functions by $\cA_{M\times S}$. We also denote the polar dual of $S$ by $S^\vee$. Then, for $c\in S^\vee$, $e^{-c/\hbar}\in \cA_{M\times S}$. The linear combinations of such elements form a ring isomorphic to the polynomial ring of $S^\vee$. With respect to the conic topology of $S^\vee$, we can complete the group ring, and we obtain the Novikov ring $\Lambda_0^S$ associated to $S$. We can also complete $\cA_{M\times S}$ so that it becomes a module over $\Lambda_0^S$. Then we take the quotient by the functions which decay rapider than every $e^{-c/\hbar}$ with $c\in S^\vee$. We denote the resulting ring by $\cOaeL_{M\times S}$.

We will also use some variants of $\cOaeL$; the hierarchy is as follows:
\begin{equation}
    \cOaeL\subset \cOceL\subset \cOseL\subset \cOdeL.
\end{equation}
The role of each ring is as follows:
\begin{enumerate}
    \item $\cOaeL$ will be used to work on actual asymptotic analysis.
    \item $\cOceL$ will be used to formulate the summability and all the modules we treat essentially come from this level.
    \item $\cOseL$ naturally appears when we consider a canonical filtration on $\cOaeL$.
    \item $\cOdeL$ is used as the base ring of our categories, which is a slight ramification of $\cOseL$.
\end{enumerate}

Now the category of the differential equation side is defined to be the bounded derived category of summable $\cDceL_{M\times S}$-modules. It roughly consists of the following conditions:
\begin{enumerate}
    \item $\cDceL_{M\times S}:=\cOceL_{M\times S}\otimes \cD^{\hbar}_{M\times S}$-module structure,
    \item Coherency,
    \item Holonomicity, which says the characteristic variety is Lagrangian, 
    \item WKB-regularity, which is a consistency condition between the formal type and the characteristic variety,
    \item Summability, which says existence of solutions with certain asymptotic behaviors.
\end{enumerate}
We denote the category by $D^b_{\mathrm{sum}}(\cDceL_{M\times S})$. Then we also obtain the locally base-changed version $D^b_{\mathrm{sum}}(\cDdeL_{M\times S})$. 

The monodromy side is sheaf quantization. Our sheaf quantization can be roughly described by the following structures and conditions:
\begin{enumerate}
    \item A sheaf quantization is a $\bC$-equivariant sheaf on $M\times \bC_t$ with an action of the base ring $\cOdeL_S$,
    \item with microsupport in $S^\vee$,
    \item whose $\Lambda^S_0$-module structure and $\cOdeL_S$-module structure are compatible,
    \item and its non-conic microsupport is algebraic Lagrangian,
    \item perfectness with respect to $\cOdeL_S$-module structure.
\end{enumerate}
We denote the resulting category by $\SQ_S(T^*M)$.

\subsection{Main results}
Now we can state the main theorem of this paper. 
\begin{theorem}
Let $S$ be a sectoroid. There exists an equivalence
\begin{equation}
    \Sol^\hbar_S\colon D^b_{\mathrm{sum}}(\cDdeL_{M\times S})\rightarrow \SQ_S(T^*M).
\end{equation}
\end{theorem}

The construction of the solution functor is much inspired from the construction of D'Agnolo--Kashiwara's Riemann--Hilbert correspondence \cite{DK}. Already in their work, the definition of the topological side (``enhanced ind-sheaves") involves the extra $\bR$. The difference between their work and our work can be said as: In \cite{DK}, despite the appearance of $\bR$, it is somehow squashed by taking ind-limit toward $+\infty$, while we keep the extra direction as it is. Also, in some sense, their work belongs to the exact symplectic world (over $\bC$), and our work belongs to the non-exact symplectic world (over the Novikov ring). When an $\hbar$-differential equation is without completion, we can compare the solution sheaf quantization and the enhanced ind-sheaf. It is described in Section 17. 

The Novikov ring also used in the Riemann--Hilbert correspondence in the author's previous paper~\cite{KuwIrreg} as a reformulation of D'Agonolo--Kashiwara's theory. The hom-spaces defined over the (finite) Novikov ring in \cite{KuwIrreg} are close relatives of our hom-spaces of sheaf quantizations. Again, since it belongs to the exact symplectic world, we could specialize the Novikov indeterminate and we could work over $\bC$. 

In the above theorem, we require the summability as an axiom. Hence it is not clear whether a given differential equation is summable or not. As a second main theorem, we give a sufficient condition for the summability.
\begin{theorem}
Any weakly semi-globally semisimple WKB-regular flat $\hbar$-connection valued in $\cOaeL_{M\times S}$ is summable at  any small $\theta\in S$.
\end{theorem}
The definition of the weak semi-global semisimplicity is given in the body of the paper. The proof is based on Wasow's proof in $\dim=1$~\cite{Wasow} and Sibuya's block-diagonalization method. Also, there are some inspirations from asymptotic liftings of solutions of irregular differential equations in higher dimensions due to Mochizuki \cite{WildTwistor} (see also \cite{SabbahIntrotoStokes, Hien}.

We would like to summarize what is written in each section. In section 2, we will collect some languages concerning the real blow-up of $\bC_\hbar$. In section 3, we collect some constructions of the Novikov ring associated to subsets of the real-blowup. Section 4 describes sheaves of functions appearing in this paper. Section 5 introduces the differential equation side of our correspondence. Section 6 explains some constructions in the theory of equivariant sheaves. Section 7 explains how the Novikov rings appear in microlocal sheaf theory. Based on this, Section 8 explains the sheaf quantization side. 

Section 9 constructs our solution functor relating the two sides of the category. In section 10, we discuss the localization of $\cD^\hbar$-modules. Section 11 explains the notion of WKB-regularity, which is a condition to make geometric data compatible with Stokes data. Section 12 explains basic facts to solve differential equations in the completed setting. Then, Section 13, explains the notion of the summability, which is in some sense an axiomatic form of a weak version of the WKB summability. 

On these preparations, we prove the equivalence statement  in Section 14. In Section 15, we give a sufficient condition for the summability. The method is a Novikov-adaptation of classical asymptotic analytic results obtained by Sibuya, Wasow and others. 

Section 16 compares the present result and the previous result in \cite{kuwagaki2020sheaf}. 
Section 17 introduces the notion of convergent lattice to speak about objects defined over $\bC$. Using this notion, we compare our construction and D'Agnolo--Kashiwara's construction~\cite{DK}. In Appendix, we state a result concerning a property of algebraic Lagrangian in the cotangent bundle, which is used in the body of the paper.

\subsection*{Acknowledgement}
I would like to thank Kohei Iwaki and Akishi Ikeda for their interests and teaching me the ideas of exact WKB and resurgent analysis. I also would like to thank Takahiro Saito for various discussions on subjects related to Rees modules, Yuichi Ike for some discussions related to $\gamma$-topology and sheaf quantizations. Special thanks to Kyoji Saito for his emphasis on Sato's philosophy prevailing on the whole story, which brings me an idea to begin the above introduction with Sato's idea.

The work is supported by MEXT KAKENHI Grant Number 18K13405, 18H01116, 20H01794.

\section{Topology of the real blow-up of $\bC_\hbar$}
We would like to prepare some terminology around the oriented real blow-up of $\bC$ which plays a fundamental role in asymptotic analysis.
\subsection{Real blow-up and sector}
Let $\bC_\hbar$ be the complex line with the standard coordinate $\hbar$. Let $\varpi\colon {\wC_{\hbar}}\rightarrow \bC_\hbar$ be the real oriented blow-up at $0$. Concretely, 
\begin{equation}
\begin{split}
      {\wC_\hbar}&:=\lc (z, \theta)\in \bC_\hbar\times \bR/\bZ\relmid  \Im(ze^{-2\pi\sqrt{-1}\theta})= 0, \Re(ze^{-2\pi\sqrt{-1}\theta})\geq 0\rc\\
      &=\lc (z, \theta)\in \bC_\hbar\times \bR/\bZ\relmid  \arg z=2\pi\theta \text{ if }z\neq 0 \rc\\
\end{split}
\end{equation}
and the map $\wC_\hbar\rightarrow \bC_\hbar$ is given by the projection $(z, \theta)\mapsto z$. The set $\varpi^{-1}(0)$ is called the exceptional divisor.

For a subset $D\subset \bC_\hbar$, we use the following notation:
\begin{equation}
    \widetilde{D}:=\mathrm{Int}\lb\overline{\varpi^{-1}(D)}\cap \varpi^{-1}(0)\rb\cup \varpi^{-1}(D)
\end{equation}
where $\mathrm{Int}$ is the interior as a subset of $\varpi^{-1}(0)$.

\begin{definition}
An open sector $S$ is a subset of ${\wC_\hbar}$ of the form that there exist $\theta_1 <\theta_2\in\bR$ and $\epsilon \in \bR_{>0}$
\begin{equation}
 {S}=\lc(z,\theta)\in \wC_\hbar\relmid 0\leq |z|<\epsilon, \theta_1<\theta<\theta_2\rc.
\end{equation}
The definition of a closed sector is obtained by replacing $\theta_1<\theta<\theta_2$ with $\theta_1\leq \theta\leq \theta_2$.

We set 
\begin{equation}
    S^\circ:=S\bs \varpi^{-1}(0).
\end{equation}
In this notation, one can interpret the following equality appropriately:
\begin{equation}
\widetilde{S^\circ}=S.
\end{equation}

\end{definition}
Sometimes it is convenient to work with a more flexible framework.
\begin{definition}
We consider the topology of $\wC_\hbar$ generated by the open sectors. We denote the topological space by $\wC^s_\hbar$. An open subset of this space is called a sectoroid.
The forgetful map $\wC_\hbar\rightarrow \wC_\hbar^s$ is denoted by $\frakf$.
\end{definition}

\begin{lemma}
Let $S$ be a nonempty sectoroid. The intersection $\varpi^{-1}(0)\cap S$ is a nonempty open subset of $\varpi^{-1}(0)$.
\end{lemma}
\begin{proof}
Only the nontrivial point is the nonemptiness. We will prove the following assertion: For any $0\leq r<1$ and any sectoroid $S$, we have $rS\subset S$. Since it is true for every sector, the assertion remains true after forming the topology.
\end{proof}

We have the projection
\begin{equation}
\frakp\colon {\wC_\hbar}\rightarrow \bR/\bZ; (z, \theta)\mapsto \theta.
\end{equation}
We call this map the polar projection. Note that the polar projection map factors through the forgetful map $\frakf\colon {\wC_\hbar}\rightarrow {\wC_\hbar^s}$. We denote the resulting map $\wC^s_\hbar\rightarrow \bR/\bZ$ again by $\frakp$ if no confusions occur.
The following is obvious.
\begin{corollary}
For a sectoroid $S$, we have
\begin{equation}
    \varpi^{-1}(0)\cap S=\frakp(S).
\end{equation}
\end{corollary}
\begin{proof}
This is a straight forward corollary of the above lemma.
\end{proof}

\begin{definition}
For a sectoroid $S$, we set
\begin{equation}
    S^\circ:=S\bs \varpi^{-1}(0).
\end{equation}
\end{definition}

We also prepare the terminology of cones. The scaling action on $\bC$ is the action of $\bR_{>0}$ on $\bC$ given by the multiplication.
\begin{definition}
Let $C$ be a subset of $\bC$. We say $C$ is a cone if it is stable under the scaling action.
\end{definition}

\begin{definition}
For a sectoroid $S$, we set 
\begin{equation}
    \Cone(S):=\lc h\in \bC_\hbar\relmid r\cdot \hbar\in \varpi(S^\circ) \text{ for some } r\in \bR_{>0}\rc,
\end{equation}
which is a cone in $\bC_\hbar$.
\end{definition}

For a cone $C$ in $\bC$, the hull $h(C)$ is defined by
\begin{equation}
    h(C):=\lc x\in \bC\relmid x=y+z \text{ for some } y, z\in C\rc.
\end{equation}
The hull $h(C)$ is again a cone and is now closed under addition.
\begin{definition}
We say a sectoroid $S$ is acute if $h(\Cone(S))$ is acute.
\end{definition}

Let $C$ be a cone of $\bC_\hbar$. We denote the dual space of $\bC_\hbar$ by $\bC^\vee_\hbar$. The polar dual $C^\vee$ is defined by
\begin{equation}
    C^\vee:=\lc a\in \bC_\hbar^\vee\relmid \Re(a\overline{b})\geq 0 \text{ for any $b\in C$}\rc
\end{equation}
where $\overline{b}$ is the complex conjugate of $b$. Note that $C^\vee$ and $h(C)^\vee$ are the same. If $S$ is not acute, the polar dual $\Cone(S)^\vee$ is the singlet $\{0\}$.

\section{Novikov rings}
\subsection{$\gamma$-complete product}
We first fix the following notation. Let $\gamma$ be a closed cone in $\bC$ defined by $\lc z\in \bC\relmid\theta_1 \leq \arg z\leq \theta_2 \rc$ for some $\theta_i\in \bR$. Let $\lc V_c\rc_{c\in \bC}$ be a set of vector spaces indexed by $\bC$. 

\begin{definition}\label{def: gammaprod}
Take $\hbar\in \mathrm{RelInt}(\gamma^\vee)$. For a subset $\frakc\subset \bC$, we say $\frakc$ is $\gamma$-finite if 
\begin{enumerate}
    \item there exists a $c\in \bC$ such that $c+\gamma$ contains $\frakc$, and 
    \item $\lc c\in \frakc\relmid  \Re(\overline{\hbar}c)\leq N\rc$ is finite for any $N\in \bR$.
\end{enumerate}
\end{definition}

It is easy to see that:
\begin{lemma}
The $\gamma$-finiteness does not depend on the choice of $\hbar\in\mathrm{RelInt}(\gamma^\vee)$.
\end{lemma}

We set
\begin{equation}
    {\prod_{c\in \gamma}}^\gamma V_c:=\lc \prod_{c\in \gamma}v_c\in \prod_{c\in \gamma}V_c \relmid \text{ the set $\lc c\in \bC\relmid v_c\neq 0\rc$} \text{ is $\gamma$-finite.} \rc
\end{equation}
If $\gamma$ is not acute, $\gamma^\vee=\{0\}$. Then we have ${\prod_{c\in \gamma}}^\gamma V_c:=\bigoplus_{c\in \gamma} V_c$

We would like to recall the notion of $\gamma$-topology introduced by Kashiwara--Schapira~\cite{KS} in our setting.

\begin{definition}[$\gamma$-topology]
A Euclidean open subset $U$ of $\bC$ is said to be $\gamma$-open if it satisfies 
\begin{equation}
    U+\gamma\subset U.
\end{equation}
The set of $\gamma$-open subsets forms a topology, which is called $\gamma$-topology.
\end{definition}
In this language, the first condition in Definition~\ref{def: gammaprod} can be rephrased as: The set $\frakc$ is compact in $\gamma$-topology.

\subsection{Novikov rings}
In this section, we introduce some versions of the Novikov ring.

Though the content in this section works over any field, we will work over the complex numbers $\bC$, which is our main interest. We consider the set of non-negative real numbers $\bR_{\geq 0}$ as a monoid by the usual addition. We denote the associated polynomial ring by $\bC[\bR_{\geq 0}]$ and the indeterminate by $T^c$ for $c\in \bR_{\geq 0}$. The Novikov ring $\Lambda_0$ is defined by
\begin{equation}
    \Lambda_0:=\lim_{\substack{\longleftarrow\\c\rightarrow \infty}}\bC[\bR_{\geq 0}]/T^c\bC[\bR_{\geq 0}].
\end{equation}
As a set, we can alternatively describe $\Lambda_0$ as the set of formal sums of the form
\begin{equation}
    \sum_{c\in \bR_{\geq 0}}a_cT^c
\end{equation}
with $a_c\in \bC$ and the set $\lc c\in \bR_{\geq 0}\relmid a_c\neq 0 \rc$ is without accumulation. The ring $\Lambda_0$ is a valuation ring. 

A slightly extended version is the complexified Novikov ring defined as follows. We set
\begin{equation}
    \bC_{\Re\geq 0}:=\lc c\in \bC\relmid \Re c\geq 0 \rc.
\end{equation}
This is also a monoid by the addition, and we can consider $\bC[\bC_{\Re\geq 0}]$. We denote the indeterminate by $T^c$ for $c\in \bC_{\Re \geq 0}$. 
Then the complexified Novikov ring $\Lambda_0^\bC$ is defined by
\begin{equation}
    \Lambda_0^\bC:=\lim_{\substack{\longleftarrow\\ \bR_{>0}\ni c\rightarrow \infty}}\bC[\bC_{\Re \geq 0}]/T^c\bC[\bC_{\Re \geq 0}].
\end{equation}
As a set, we can alternatively describe $\Lambda^\bC_0$ as a set of formal sums of the form
\begin{equation}
    \sum_{c\in \bC_{\Re \geq 0}}a_cT^c
\end{equation}
with $a_c\in \bC$ and the set $\lc c\in \bC_{\geq 0}\relmid a_c\neq 0 \rc$ is without accumulation and the map $\lc c\in \bC_{\geq 0}\relmid a_c\neq 0 \rc\rightarrow \bR_{\geq 0}$ taking real parts is a proper map.

More generally, we can define the Novikov ring associated to sectoroids as follows. Let $C$ be an open cone in $\bC$. The set
\begin{equation}
    \lc a+C^\vee\relmid a\in C^\vee\rc
\end{equation}
forms a poset by the inclusion relation. This gives the dual poset structure on the set
\begin{equation}
    \lc \bC[C^\vee]/T^c\bC[C^\vee]\relmid c\in C^\vee\rc.
\end{equation}
The limit over the latter poset is denoted by $\Lambda^C_0$. Concretely, the ring consists of formal sums of the form
\begin{equation}
    \sum_{c\in C^\vee}a_cT^c
\end{equation}
with $a_c\in \bC$ and the set $\lc c\in \bC_{\geq 0}\relmid a_c\neq 0 \rc$ is without accumulation and the map $\lc c\in \bC_{\geq 0}\relmid a_c\neq 0 \rc\rightarrow \bR_{\geq 0}$ defined by $a\mapsto \Re(a\overline{b})$ is proper for any $b\in \mathrm{RelInt}(h(C))$. Using the $\gamma$-complete product, we can write it as $\Lambda_0^C=\displaystyle{{\prod_{c\in \bC}}^{C^\vee}\bC}$.

\begin{example}
Take $C:=\bC_{\Re> 0}$. Then the dual cone is $C^\vee=\bR_{\geq 0}$. In this case,
\begin{equation}
    \Lambda_0^{\bC_{\Re> 0}}\cong \Lambda_0.
\end{equation}
\end{example}

\begin{example}
If $h(C)$ is not acute, $C^\vee=0$. Hence $\Lambda_0^C\cong \bC$.
\end{example}

\begin{example}
We would like to extend the defintion for the case when $C:=\bR_{\geq 0}$. The dual cone is $\bC_{\Re\geq 0}$. We set
\begin{equation}
    \Lambda_0^{\bR_{\geq 0}}:= \Lambda_0^\bC.
\end{equation}
\end{example}

\begin{definition}
For a sectoroid $S$, we set
\begin{equation}
    \Lambda_0^S:=\Lambda_0^{\Cone(S)}.
\end{equation}
This is an integral domain, the quotient field is denoted by $\Lambda^S$. For a sectoroid $S$, we will sometimes use the notation $S^\vee:=\Cone(S)^\vee$.
\end{definition}

\subsection{The sheaf of Novikov rings}
Let $\Open(\wC^s_\hbar)$ be the set of sectoroids. We consider the assignment
\begin{equation}
    \frakL_0\colon \Open(\wC^s_\hbar)\ni S \mapsto \Lambda^{S}_0 \in \mathrm{Ring}
\end{equation}
where $\mathrm{Ring}$ is the category of rings.
For an open inclusion $S\hookrightarrow S'$, there exists an associated inclusion $\Cone(S')^\vee\subset \Cone(S)^\vee$, which induces a ring homomorphism $\Lambda_0^{S'}\rightarrow \Lambda_0^{S}$. Hence the above assignment is a presheaf. It is easy to see it is actually a sheaf.

\begin{definition-lemma}
By the above maps, the assignment $\frakL_0$ forms a sheaf. We call this sheaf the Novikov sheaf.
\end{definition-lemma}

\section{Setup of structure sheaves}
In this section, we introduce several sheaves of functions coming from asymptotic analysis. The hierarchy is as follows:
\begin{equation}
    \cOae_{M\times S}\subset \cOc_{M\times S}\subset \cOse_{M\times S} \subset \cOexp_{M\times S}.
\end{equation}
We also have the completed version of the hierarchy
\begin{equation}
    \cOaeL_{M\times S}\subset \cOceL_{M\times S}\subset \cOseL_{M\times S}\subset \cOexpL_{M\times S}.
\end{equation}

\subsection{The largest}
For a complex manifold $X$, we denote the structure sheaf by $\cO_X$ i.e., the sheaf of holomorphic functions.

Let $M$ be a complex manifold. We have the following diagram:
\begin{equation}
    \xymatrix{
M\times (\wC_\hbar\bs \varpi^{-1}(0))\ar[r]^-\iota \ar[d]^{\id \times \varpi} & M\times \wC_\hbar \ar[d]^{\id\times \varpi} \\
M\times (\bC_\hbar\bs {0})\ar[r]_{}& M\times \bC_\hbar &.
}
\end{equation}
where the horizontal arrows are inclusions.

We set
\begin{equation}
    \cO_{M\times\widetilde{\bC}_\hbar}:=\iota_*(\id\times \varpi)^{-1}\cO_{M\times (\bC_\hbar\bs 0)}.
\end{equation}
Since it is defined by the push-forward, any section of $\cO_{M\times \wC_\hbar}$ can be regarded as a holomorphic function on an open subset $V$ of $M\times (\bC_\hbar\bs 0)$. We use this point of view freely without mentioning. In particular, for a sectoroid $S$ and an open subset $U$ in $M$, a section of $\cO_{M\times \wC_\hbar}$ over $S\times U$ is interpreted as a holomorphic function on $S^\circ \times U$.

We set
\begin{equation}
\begin{split}
    \cO_{M\times \wC_\hbar^s}&:=\frakf_*\cO_{M\times \wC_\hbar}.
\end{split}    
\end{equation}

\subsection{The sheaf $\cA$ and $\cOae$}
We would like to recall the sheaf $\cA$ in $\cite{Majima, SabbahIntrotoStokes}$. 

Let $C^{\infty}_{\wC_\hbar}$ be the sheaf of $C^\infty$-functions on $\wC_\hbar$.
On $\bC_\hbar$, we consider the real coordinate given by $re^{i\theta}$. Then $\overline{\hbar}\partial_{\overline{\hbar}}=\frac{1}{2}\lb r\partial_r+\sqrt{-1}\partial_\theta\rb$ can be lifted to act on $C^{\infty}_{\wC_\hbar}$. Then $\cA_{ \wC_\hbar}$ is locally defined by
\begin{equation}
    \cA_{\wC_\hbar}:=\ker(\overline{\hbar}\partial_{\overline{\hbar}}).
\end{equation}
This is a subsheaf of $\cO_{\wC_\hbar}$. We set $\cA_{\wC_\hbar^s}:=\frakf_*\cA_{\wC_\hbar}$. Then the sheaf admits the asymptotic expansion morphism
\begin{equation}
    \cA_{\wC_\hbar^s}\rightarrow \bC[[\hbar]].
\end{equation}
If a function becomes zero after the asymptotic expansion, we say the function is rapid decay.

We can similarly define $\cA_{M\times \wC_\hbar}$ by the kernel of $\overline{\hbar}\partial_{\overline{\hbar}}$ in $C^\infty_{M\times \wC_\hbar}$. We can also define $\cA_{M\times \wC_\hbar^s}$.

\begin{lemma}
Let $S$ be a sectoroid.
 For $c\in S^\vee$, we have $e^{-c/\hbar}$ is in $\cA_{M\times \wC^s_\hbar}(M\times S)$. For any $\psi\in \cA_{M\times \wC^s_\hbar}(U)$ for an open subset $U\subset M\times \wC^s_\hbar$ and $c\in S^\vee$, the product $e^{-c/\hbar}\psi$ is asymptotically expanded to $0$. 
\end{lemma}
\begin{proof}
For $0\neq c\in S^\vee$ and $b\in S$, we have $\Re(c\overline{b})>0$. Equivalently, $\Re(c/b)>0$. Hence $-\Re(c/\hbar) <0$ for $c\in S^\vee$ and $\hbar\in S$.
\end{proof}

We set $\cA_{M\times S}:=\cA_{M\times  \wC^s_{\hbar}}|_{M\times S}$. From here, we equip $S$ with the induced topology from $\wC_\hbar^s$.

\begin{definition}
A section $\psi\in \cA_{M\times S}(U)$ is said to be very rapid decay if $e^{c/\hbar}\psi\in \cA_{M\times S}(U)$ for any $c\in S^\vee$. The very rapid decay functions form an ideal subsheaf of $\cA_{M\times S}$. We denote it by $\cO^{vr}_{M\times S}$.
\end{definition}

We set
\begin{equation}
    \cO^{\ae}_{M\times S}:=\cA_{M\times S}/\cO_{M\times S}^{vr}.
\end{equation}

\subsection{Novikov completion $\cOaeL$}
We would like to introduce a decreasing filtration on $\cA_{M\times S}:=\cA_{M\times  \wC^s_{\hbar}}|_{M\times S}$. From here, we equip $S$ with the induced topology from $\wC_\hbar^s$.

Let $\frakD$ be the set of bounded open neighborhoods of $0$ in the dual space $\bC_\hbar^\vee$ of $\bC_\hbar$. We set
\begin{equation}
    F^{\leq D}\cA_{M\times  S}(U):=\lc f\in \cA_{M\times  S}(U)\relmid e^{r/\hbar}\psi \in \cA_{M\times  S}(U)\text{ for any $
    r\in D\cap S^\vee$}\rc
\end{equation}
for $D\in \frakD$ and an open subset $U\subset M\times S$. In this notation, $\cO^{vr}_{M\times S}$ is $F^{\leq \bC_\hbar^\vee}\cOae_{M\times S}$.
Hence the filtration induces a filtration on ${\cO}^{\ae}_{M\times S}$. We denote it by $F^{\leq D}\cOae_{M\times S}$.
We set
\begin{equation}
    \cOaeL_{M\times S}(U):=\lim_{\substack{\longleftarrow \\ D\in \frakD}}\cO^{\ae}_{M\times S}(U)/F^{\leq D}\cO^{\ae}_{M\times S}(U).
\end{equation}
This forms a sheaf on $M\times S$. By the definition, an element of this set can be represented as
\begin{equation}
    \sum_{i\in \bZ_{\geq 0}}e^{-r_i/\hbar}\psi_i
\end{equation}
where $r_i\in S^\vee$ and $\psi_i\in \cA_{M\times S}$ and the set $\lc r_i\relmid \psi_i\neq 0\rc$ is finite inside a compact region.

We also sometimes use the notation
\begin{equation}
    \Gr^{\leq D}\cOaeL_{M\times S}(U):= \Gr^{\leq D}\cOae_{M\times S}(U):= \cOae_{M\times S}(U)/F^{\leq D} \cOae_{M\times S}(U)
\end{equation}
\begin{remark}
From this expression, we can consider $\cOaeL_{M\times S}$ as a version of the ring of transseries~\cite{Ecalle, Costin}. However, they are certainly different, in the sense that we do not have a unique series expression.
\end{remark}

The subset of $\cOaeL_{M\times S}$ spanned by $\lc e^{-c/\hbar}\relmid c\in S^\vee\rc$ forms a ring and isomorphic to $\Lambda_0^S$. In particular, 
\begin{lemma}
The sheaf $\cOaeL_{M\times S}$ is a module over $\Lambda_0^S$.
\end{lemma}

One can organize this module structure over $\wC_\hbar^s$ if one wants: Consider the category of $\cOaeL_{M\times S}$-modules $\Mod(\cOaeL_{M\times S})$. This forms a category of sheaves over $\wC_s$ we denote it by $\Mod(\cOaeL_{M\times \wC^s_\hbar})$.
\begin{lemma}
The sheaf $\Mod(\cOaeL_{M\times \wC^s_\hbar})$ is a module over $\frakL_0$.
\end{lemma}

\subsection{Exponential version}
\begin{definition}
For an open subset $U\subset M\times S$, we say a section $f\in \cO_{M\times S}(U)$ is of exponential growth if there exists $r\in S^\vee$ such that $e^{-r/\hbar}f\psi\in \cA_{M\times S}(U)$. We denote the sheaf spanned by such sections by $\cA^{\exp}_{M\times S}$. 
\end{definition}

We set
\begin{equation}
    \cOexp_{M\times S}:=\cA^{\exp}_{M\times S}/ \cO^{vr}_{M\times S}.
\end{equation}

Of course, $\cOexp_{M\times \wC^s_\hbar}$ contains $\cO^{\ae}_{M\times \wC^s_\hbar}$ as a subsheaf.

We can similarly define a filtration:
Let $\frakD$ be the set of bounded open neighborhoods of $0$ in the dual space $\bC_\hbar^\vee$ of $\bC_\hbar$. We set
\begin{equation}
    F^{\leq D}\cA^{\exp}_{M\times  S}(U):=\lc \psi\in \cA^{\exp}_{M\times  S}(U)\relmid e^{r/\hbar}\psi \in \cA_{M\times  S}(U)\text{ for any $
    r\in D\cap S^\vee$}\rc
\end{equation}
for $D\in \frakD$ and an open subset $U\subset M\times S$. Again, this induces a filtration on $\cO^{\exp}_{M\times S}$.
We set
\begin{equation}
    \cO^{\Lambda}_{M\times S}(U):=\lim_{\substack{\longleftarrow \\ D\in \frakD}}\cOexp_{M\times S}(U)/F^{\leq D}\cOexp_{M\times S}(U).
\end{equation}
This forms a sheaf on $M\times S$. Similar to the previous case, this sheaf is a module over the Novikov field $\Lambda^S$.
\begin{lemma}
The sheaf $\cO^{\Lambda}_{M\times S}$ is a module over $\Lambda^S$.
\end{lemma}

\subsection{Subexponential growth functions}
In our situation, we may not have transseries expressions. In such a case, $\cOae$ is not so a natural place to consider. Instead, we consider the space of subexponential functions.

Let $S$ be a sectoroid. We define the space of subexponential functions by
\begin{equation}
    \cA_{M\times S}^{\mathrm{se}}(U):=\lc \psi\in \cA_{M\times S}^{\exp}(U)\relmid e^{-c/\hbar}\psi\in \cA_{M\times S}(U) \text{ for any } c\in S^\vee\bs \{0\}\rc,
\end{equation}
which contains $\cA_{M\times S}(U)$ as a subspace.

\begin{lemma}
The module $\cA_{M\times S}^{\mathrm{se}}$ forms a subring of $\cA_{M\times S}^{\exp}$.
\end{lemma}
\begin{proof}
For any $\psi, \phi\in  \cA_{M\times S}^{\mathrm{se}}(U)$, we have $e^{-c/\hbar}(\psi+\phi)=e^{-c/\hbar}\psi+e^{-c/\hbar}\phi\in \cA_{M\times S}(U)$ for any $c\in S^\vee\bs 0$, since $\psi, \phi\in \cA^\se_{M\times S}(U)$. We also have $e^{-c/\hbar}\psi\phi=(e^{-c/2}\psi)(e^{-c/2}\phi)\in \cA^{\se}_{M\times S}(U)$ by $\psi, \phi\in \cA^\se_{M\times S}(U)$. This completes the proof.
\end{proof}

We also set
\begin{equation}
    \cOse_{M\times S}:=\cA_{M\times S}^{\se}/ \cO^{vr}_{M\times S}.
\end{equation}

\begin{lemma}
The space $\cO_{M\times S}^{vr}$ is an ideal of $\cA_{M\times S}^{\se}$. Hence $\cOse_{M\times S}$ is a ring.
\end{lemma}
\begin{proof}
Obvious.
\end{proof}

Similarly, we can define $\cOseL_{M\times S}$ as a completion.

\subsection{Subexponential support}
For $\psi\in \cOexpL_{S}(S)$, we set
\begin{equation}
    \supp_\se(\psi):=\Int\lc c\in \bC\relmid  \psi e^{c/\hbar}\in \cOseL_S(S)\rc.
\end{equation}
\begin{lemma}
The subset $\supp_\se(\psi)$ is a $\gamma$-open subset for $\gamma=-S^\vee$.
\end{lemma}
\begin{proof}
This is an obvious consequence of the definition of $-S^\vee$.
\end{proof}

\begin{lemma}
For any $\psi\in \cOseL_S(S)$ and $\phi\in \cOexpL_S(S)$, we have $\supp_{\se}(\phi)\subset\supp_{\se}(\psi\phi)$.
\end{lemma}
\begin{proof}
This is also obvious.
\end{proof}

\begin{lemma}
For $\psi\in \cOexpL_S(S)$, suppose $\supp_{\se}(\psi)\supset -S^\vee$. Then $\psi\in \cOseL_S(S)$.
\end{lemma}
\begin{proof}
This also follows from the definition of $\cOseL_S$.
\end{proof}

\subsection{Continuously extendable functions}
We set
\begin{equation}
    \cA^{\mathrm{ce}}_{M\times S}(U):= \lc \psi\in \cA^{\exp}_{M\times S}(U)\relmid \psi|_{\hbar=0} \text{ is well-defined and defines a holomorphic function on $M$}\rc.
\end{equation}
Note that $\cA^{\mathrm{ce}}_{M\times S}\subset \cA_{M\times S}^{\se}$. We also set
\begin{equation}
    \cOc_{M\times S}:=\cA^{\mathrm{ce}}_{M\times S}/ \cO^{vr}_{M\times S}.
\end{equation}
We can again complete and get $\cOceL_{M\times S}$.

\subsection{Moderate growth around poles}
Let $D$ be a normal crossing divisor in a complex manifold $U\subset \bC^n$.
Let $\varpi_D\colon \widetilde{U}\rightarrow U$ be the oriented real blow-up of $U$ at $D$.
As in \cite{SabbahIntrotoStokes}, we consider $\cA^{\mathrm{mod}D}_{\widetilde{U}}$ be the sheaf of moderate growth functions. It is known that $\varpi_{D*}\cA_{\widetilde{U}}^{\mathrm{mod} D}\cong \cO_U(*D)$.

For $\bullet\in \{{\ae}, \mathrm{ce}, \se, {\exp}\}$ or $\bullet\in \{{\ae}, \mathrm{ce}, \se, {\exp}\}\times \{\Lambda_0\}$, 
we set
\begin{equation}
    \cO_{\widetilde{U}\times S}^{\bullet, \mathrm{mod}D}:=\cA_{\widetilde{U}}^{\mathrm{mod}D}\otimes_{\pi^{-1}\cO_U}\varpi_D^{-1}\cO_{U\times S}^{\bullet}.
\end{equation}

\section{$\cD^\hbar$-modules}
Sheaf-theoretic treatments of $\hbar$-differential operator in the context of asymptotic analysis was first discussed by \cite{AKKT} and then formulated in general by Polesello--Schapira~\cite{PS}. Related materials are discussed in \cite{MHMproject} in the context of Rees modules. We collect some basic facts for our purpose.

\subsection{Basic notions}
Let $M$ be a complex manifold. Let $\bC_\hbar$ be the complex line with the standard coordinate $\hbar$. Let $\cD_{M\times \bC_\hbar/\bC_\hbar}$ be the sheaf of relative differential operators. Inside the ring, we consider the subring generated by $\hbar\partial, \cO_M, \hbar$ and denote it by $\cD^\hbar_M$. When we want to emphasize that the $\hbar$-independence is polynomial, we write it as $\cD^{\hbar, \mathrm{al}}_M$. This is a sheaf over $M\times \bC_\hbar$. 

The ring $\cD^\hbar$ has many names: In the context of twistor modules(e.g.~\cite{mochizukiMTM}), it is called $\cR$. In the project \cite{MHMproject}, it is called $\widetilde{\cD}$. We use the notation $\cD^\hbar$ to emphasize its relation to deformation quantization~\cite{KSDQ}:
\begin{lemma}
The quotient $\cD^\hbar_M/\hbar\cD_M^\hbar$ is isomorphic to the push-forward of the structure sheaf $\cO_{T^*M}$ of $T^*M$ to $M$ where we consider the fiber direction as algebraic.
\end{lemma}
For the economics of the notation, we denote $\cD^\hbar_M/\hbar\cD^\hbar_M$ by $\cO_{T^*M}$.

We will use various base-change of $\cD^\hbar_M$. In the previous section, we introduce various sheaves of functions. For a sectoroid $S,$ as a general notation, we set
\begin{equation}
    \cD^{\bullet}_{M\times S}:=\cO^{\bullet}_{M\times S}\otimes_{\varpi^{-1}\cO_{M}[\hbar]}\frakf_*\varpi^{-1}\cD^{\hbar}_{M}
\end{equation}
where $\bullet\in \{{\ae}, \mathrm{ce}, \se, {\exp}\}$ or $\bullet\in \{{\ae}, \mathrm{ce}, \se, {\exp}\}\times \{\Lambda_0\}$. 
For example, we have
\begin{equation}
    \cDceL_{M\times S}:=\cOceL_{M\times S}\otimes_{\varpi^{-1}\cO_{M}[\hbar]}\frakf_*\varpi^{-1}\cD^{\hbar}_{M}.
\end{equation}
We have already seen that $\cOceL_{M\times S}$ contains $\Lambda_0^S$. We denote the maximal ideal of $\Lambda_0^S$ by $\frakm_S$. We also denote the subsheaf of $\cOc_{M\times S}$ generated by the functions vanishing at $\hbar=0$ by $I_0$.
Let $\frakm$ be the subsheaf of $\cOceL_{M\times S}$ generated by $\frakm_S$ and $I_0$. Then the quotient $\cOceL_{M\times S}/\frakm$ is isomorphic to $\cO_M$. Similarly, we have
\begin{lemma}
We have
\begin{equation}
\begin{split}
     \cDceL_{M\times S}/\frakm \cDceL_{M\times S}&\cong \cO_{T^*M}\\
     \cDaeL_{M\times S}/\frakm \cDaeL_{M\times S}&\cong \cO_{T^*M}\\
\end{split}
\end{equation}
For the second line, $\frakm$ means $\frakm\cap \cDaeL_{M\times S}$.
\end{lemma}

In this paper, a $\cD^{\bullet}_{M\times S}$-module means a left $\cD^{\bullet}_{M\times S}$-module. We will denote the bounded derived category of coherent $\cD^{\bullet}_{M\times S}$-modules by $D^b(\cD^{\bullet}_{M\times S})$. The subcategory spanned by cohomologically coherent modules will be denoted by $D^b_{\coh}(\cD^{\bullet}_{M\times S})$. When working with this category, all the operations will be derived without any particular notations.

\begin{remark}
Since $\Lambda_0^S$ is not Noetherian, the coherence is not equivalent to finite generation. In particular, a trivial module $\bC\cong \Lambda_0^S/\frakm_S$ is not coherent. 
\end{remark}

\subsection{Modules and Spectral variety}
For a $\cDaeL_{M\times S}$-module or $\cDceL_{M\times S}$-module $\cM$, we set
\begin{equation}
    \cM^{cl}:=\cM/\frakm\cM
\end{equation}
Then $\cM^{cl}$ is a module over $\cO_{T^*M}$. We set
\begin{equation}
    \Char(\cM):=\supp(\cM^{cl})\subset T^*M.
\end{equation}
We call it the characteristic variety of $\cM$.
\begin{remark}
One may prefer the term {\em spectral variety}, regarding the context of Higgs bundles.
\end{remark}

\begin{definition}
If $\cM$ is a coherent module and $\Char(\cM)$ is a Lagrangian subvariety, we say $\cM$ is a holonomic module. We will denote the subcategories of cohomologically holonomic modules by $D^b_{\hol}(\cDaeL_{M\times S})$ and $D^b_{\hol}(\cDceL_{M\times S})$
\end{definition}

Sometimes, it is useful to impose the torsion-freeness.
\begin{definition}
For a $\cDaeL_{M\times \wC^s_\hbar}$-module $\cM$ is strict if $\bC[\hbar]$-action is free (resp. $\frakn$-torsion free, resp, $\frakm$-torsion free).
\end{definition}

\begin{remark}
Suppose $\cM$ is strict. From $\cM$, we can cook up a $\cW_X(0)$-module~\cite{PS}. Then the support as a $\cW_X(0)$-module is the same as $\Char(\cM)$.
\end{remark}

\begin{remark}
To define characteristic varieties for $\cDseL$-modules, we have to microlocalize $\cDseL$, which we will not develop here.
\end{remark}

\subsection{Basic operations}
We collect here some basic operations. Since the proofs are standard, we omit them.

\subsubsection*{Canonical module}
Let $\Omega_M$ be the canonical sheaf of $M$. This is a right $\cD_M$-module by Lie derivations. We denote the projection $M\times S\rightarrow M$ by $p_\hbar$.

We set
\begin{equation}
    \Omega_{M\times S}^{\bullet}:=\cO^{\bullet}_{M\times S}\otimes_{p_\hbar^{-1}\cO_M} p_\hbar^{-1}\Omega_M.
\end{equation}
Then this is a right $\cD^{\bullet}_{M\times S}$-modules.

\subsubsection*{Side-change}
As usual, we have the side-changing functor
\begin{equation}
    \Omega^{\bullet}_{M\times S}\otimes (-)\colon D^b(\cD^{\bullet}_{M\times S})\rightarrow D^b((\cD^{\bullet}_{M\times S})^{\mathrm{op}}).
\end{equation}

\subsubsection*{Pull back}
Let $f\colon M\rightarrow N$ be a morphism between complex manifolds. The transfer module is defined by
\begin{equation}
    \cD^{\bullet}_{M\rightarrow N}:=\cO^{\bullet}_{M\times S}\otimes_{f^{-1}\cO^\bullet_{N\times S}}f^{-1}(\cD^{\bullet}_N).
\end{equation}
This is a bimodule over $(\cD^{\bullet}_{M\times S},f^{-1}\cD^{\bullet}_{N\times S})$. For a $\cD^{\bullet}_N$-module $\cM$, the pull-back along $f$ is defined by
\begin{equation}
    _{\cD}f^*\cM:=\cD^{\bullet}_{M\rightarrow N}\otimes_{f^{-1}\cD^{\bullet}_{N\times S}}f^{-1}\cM
\end{equation}

\subsubsection*{Push forward}
We next discuss push-forwards. We define the other transfer module:
\begin{equation}
    \cD^{\bullet}_{M\leftarrow N}:=f^{-1}\cD^{\bullet}_{N\times S}\otimes_{f^{-1}\cO^{\bullet}_{N\times S}}( \Omega_{M/N}\otimes_{\cO_{M}}\cO^{\bullet}_{M\times S})
\end{equation}
which is a $(f^{-1}\cD^{\bullet}_{N\times S}, \cD^{\bullet}_{M\times S})$-module. 
For a $\cD^{\bullet}_{M\times S}$-module $\cM$, we set
\begin{equation}
    _{\cD}f_*\cM:=f_*(\cD^{\bullet}_{M\leftarrow N}\otimes_{\cD^{\bullet}_{M\times S}} \cM).
\end{equation}

\subsubsection*{Tensor products}
We also define tensor products. Let $\delta\colon M\rightarrow M\times M$ be the diagonal map. For $\cM, \cN$ be two $\cD^{\bullet}_{M\times S}$-modules. 
\begin{equation}
    \cM\overset{\cO}{\otimes}\cN:={}_\cD\delta^{*}(\cM\boxtimes \cN)
\end{equation}
where $\delta$ is the diagonal. If $\cM$ is a usual $\cD^{\bullet}_{M\times S}$-module and $\cN$ is a right $\cD^{\bullet}_{M\times S}$-module, we can also define
\begin{equation}
    \cN\overset{\cD}{\otimes}\cM,
\end{equation}
which has no $\cD$-module structures. Similarly, we can define $\cHom_{\cD}(\cM, \cN)$ and $\cHom_\cO(\cM, \cN)$. 

\subsubsection*{Adjunction}
In this paper, we do not develop the notion of goodness for $\cD^{\bullet}_{M\times S}$-modules. Hence the adjunction statements are weaker than usual. Since all the proofs are standard, we omit them.
\begin{lemma}
Suppose $f$ is a closed embedding on $\supp(\cM)$. Then
\begin{equation}
    \Hom_{D^b(\cD^{\bullet}_{N\times S})}({}_{\cD}f_*\cM, \cN)\cong \Hom_{D^b(\cD^{\bullet}_{M\times S})}(\cM, _{\cD}f^*\cN).
\end{equation}
\end{lemma}
\begin{lemma}
Suppose $f$ is smooth. Then
\begin{equation}
    \Hom_{D^b(\cD^{\bullet}_{N\times S})}(\cM, {}_{\cD}f_*\cN)\cong \Hom_{D^b(\cD^{\bullet}_{N\times S})}({}_{\cD}f^*\cM, \cN).
\end{equation}
\end{lemma}

\begin{lemma}[base-change]
Suppose the morphisms involved in the diagram below are all embeddings:
\begin{equation}
\xymatrix{
X\ar[r]^{\iota_2} \ar[d]^{\iota_1} & Y \ar[d]^{\iota_3} \\
Y'\ar[r]_{\iota_4}& Z 
}
\end{equation}
where the diagram is Cartesian. Then we have an isomorphism
\begin{equation}
    {}_{\cD}\iota_3^*{}_{\cD}\iota_{4*}\cM\cong {}_{\cD}\iota_{2*}{}_{\cD}\iota_1^*\cN
\end{equation}
\end{lemma}

\subsubsection*{Duality}
We set
\begin{equation}
    (\Omega_M^{-1})^{\bullet}:=\Omega_M^{-1}\otimes \cO^{\bullet}_M
\end{equation}
For a $\cD^{\bullet}_{M\times S}$-module $\cM$, $\cHom_{\cD^{\bullet}_{M\times S}}(\cM, \cD^{\bullet}_{M\times S})$ is a right module. By the side-change functor, we obtain a left module
\begin{equation}
    _{\cD}\bD(\cM):=\cHom_{\cD^{\bullet}_{M\times S}}(\cM, \cD^{\bullet}_{M\times S})\otimes (\Omega_M^{-1})^{\bullet}.
\end{equation}
It is standard to see $({}_\cD\bD)^2=\id$ on $D^b_{\coh}(\cD^{\bullet}_{M\times S})$. This is the duality functor.

We set ${}_\cD f^!:={}_{\cD}\bD\circ {}_{\cD}f^{-1}{}_{\cD}\bD$.

\subsubsection*{Duality interpretation}
Following \cite{DK}, we consider the duality interpretation. Let $\delta\colon M\hookrightarrow M\times M$ be the diagonal embedding. We set $\cB_{\Delta}:={}_{\cD}\delta_*\cO^{\bullet}_{M\times S}$.
We first remark the following isomorphism.
\begin{lemma}
\begin{equation}
\begin{split}
    \Hom_{D^b(\cD^{\bullet}_{M\times S})}(\cM, \cM')&\cong \Hom_{D^b(\cD^{\bullet}_{M^3\times S})}(\cB_\Delta\boxtimes \cM, \cM\boxtimes \cB_\Delta)\\
    \Hom_{D^b(\cD^{\bullet}_{M\times S})}(\cM, \cM')&\cong \Hom_{D^b(\cD^{\bullet}_{M^3\times S})}(\cM\boxtimes\cB_\Delta, \cB_{\Delta}\boxtimes \cM)
\end{split}
\end{equation}
\end{lemma}
\begin{proof}
The proof of \cite{DK} is given by using only adjuncitons and base change. We have stated the corresponding statements in the above. So this statement also holds.
\end{proof}

\begin{lemma}
For $\cM, \cN$, we have $\cM\cong \bD\cN$ if and only if there exists two morphisms
\begin{equation}
    \begin{split}
 \cM\boxtimes\cN&\xrightarrow{\epsilon}\cB_{\Delta}[d]\\
\cB_{\Delta}[-d]&\xrightarrow{\eta} \cN\boxtimes \cM
    \end{split}
\end{equation}
such that the composition 
\begin{equation}
    \cM\boxtimes \cB_{\Delta}[-d]\xrightarrow{\id\boxtimes \eta}\cM\boxtimes\cN\boxtimes\cM\xrightarrow{\epsilon\boxtimes\id}\cB_{\Delta}[d]\boxtimes\cM
\end{equation}
is identified with $\id$ in the above lemma and the composition
\begin{equation}
    \cB_\Delta[-d]\boxtimes\cN\xrightarrow{\eta\boxtimes\id}\cN\boxtimes\cM\boxtimes\cM\xrightarrow{\epsilon\boxtimes\id}\cB_{\Delta}[d]\boxtimes\cM
\end{equation}
is identified with $\id$ in the above lemma.
\end{lemma}
\begin{proof}
The proof of \cite{DK} is again a formal manipulation, which is applicable to our situation.
\end{proof}

\subsection{Germ version}
We also define the germ version of our category. Let us fix $\theta\in \varpi^{-1}(0)$. We define the category $D^b_{\hol}(\cD^{\bullet}_{M\times \theta})$ as
\begin{itemize}
    \item $\mathrm{Obj}(D^b_{\hol}(\cD^{\bullet}_{M\times \theta}))=\bigcup_{\theta\in S}\mathrm{Obj}(D^b_{\hol}(\cD^{\bullet}_{M\times S}))$
    \item For $\cE_i\in \mathrm{Obj}(D^b_{\hol}(\cD^{\bullet}_{M\times S_i}))\subset \bigcup_{\theta\in S}\mathrm{Obj}(D^b_{\hol}(\cD^{\bullet}_{M\times S}))$, we set
    \begin{equation}
        \Hom_{D^b_{\hol}(\cD^{\bullet}_{M\times \theta})}(\cE_1, \cE_2)=\Hom_{D^b_{\hol}(\cD^{\bullet}_{M\times(S_1\times S_2)}}(\cE_1|_{S_1\cap S_2}, \cE_2|_{S_1\cap S_2}))\otimes_{\cOL_{S_1\times S_2}}\cO^{\bullet}_{\theta}
    \end{equation}
    where $\cO^{\bullet}_\theta:=\displaystyle{\lim_{\substack{\longrightarrow \\ S\ni \theta}}\cO^{\bullet}_{S}}$.
\end{itemize}

\subsection{Examples: Flat connections}
We would like to introduce some examples of modules. 

Let $\alpha$ be a global holomorphic function. For a local coordinate $x_1,..., x_n$, we consider the ideal generated by
\begin{equation}
    \lc \hbar\partial_{x_i}-\partial_{x_i}\alpha\relmid i=1,...,n\rc.
\end{equation}
This ideal is globally well-defined and defines an ideal in $\cD^{\bullet}_{M\times S}$. We denote the ideal by $\cD^{\bullet}_{M\times S}(\hbar\partial-d\alpha)$. Then one can define
\begin{equation}
    \cE^{f/\hbar}:=\cD^{\bullet}_{M\times S}/\cD^{\bullet}_{M\times S}(\hbar\partial-d\alpha).
\end{equation}

\begin{remark}
For the usual $\cD$-module case, a similar construction for meromorphic functions work well. However, for $\cD^\hbar$-modules, the situations is a bit more complicated as we shall see later.
\end{remark}

The following lemma is straight forward.
\begin{lemma}
$\Char(\cE^{\alpha/\hbar})=\mathrm{Graph}(d\alpha)$.
\end{lemma}

This module admits another interpretation.
\begin{definition}
Let $\cE$ be a locally free $\cO^{\bullet}_{M\times S}$-module. 
A flat $\hbar$-connection on $\cE$ is a linear morphism
\begin{equation}
    \nabla\colon \cE\rightarrow \cE\otimes_{\cO_{M\times S}}\Omega_{M\times S/S}
\end{equation}
satisfying 
\begin{equation}
    \nabla(\psi s)=\hbar s\otimes d\psi+\psi\nabla s
\end{equation}
for any $\psi \in \cO^{\bullet}_{M\times S}$.
\end{definition}
By setting
\begin{equation}
    \hbar\partial_{x_i}s:=\nabla_{x_i}s,
\end{equation}
we can associate a $\cD^{\bullet}_{M\times S}$-module.

Let us define a flat $\hbar$-connection $\nabla^\alpha$ on $\cO^{\bullet}_{M\times S}$ by
\begin{equation}
    \hbar\nabla^\alpha_{\partial_{x_i}}s:=\hbar \partial_{\partial_{x_i}}s+\partial_{x_i}\alpha s
\end{equation}
\begin{proposition}
The $\cD^\bullet_{M\times S}$-module defined by $\nabla^\alpha$ is isomorphic to $\cE^{\alpha/\hbar}$.
\end{proposition}
\begin{proof}
It is clear that the canonical inclusion $\cO^{\bullet}_{M\times S}\hookrightarrow \cE^{f/\hbar}$ is an isomorphism. The comparison of actions is also straight forward.
\end{proof}

The following is standard.
\begin{lemma}
Let $\cE$ be a flat $\hbar$-connection valued in $\cOceL_{M\times S}$. Then the characteristic variety is a ramified covering i.e., finite over $M$.
\end{lemma}
\begin{proof}
Take a local frame of $\cE$ as $e_1,..., e_n$. Then we have
\begin{equation}
    \nabla e_i=\sum_{j}\omega_{ij}e_j.
\end{equation}
Since $\omega_{ij}$ is in $\cOceL_{M\times S}$, we can specialize it at $\hbar=0$. Then the characteristic variety is the zero of the characteristic polynomial of $\{\omega_{ij}|_{\hbar=0}\}$. This completes the proof.
\end{proof}

\section{Equivariant sheaves and some operations}
\subsection{Equivariant sheaves}
We prepare some equivariant sheaf theory. 
Let $\cR$ be the base ring and $M$ be a complex manifold. 

Let $\bC_t$ be the complex line with the standard coordinate $t$. We consider another complex line $\bC_a$ as an additive group with discrete topology. The group $\bC_a$ acts on $\bC_t$ continuously by the translation.

We denote the category of equivariant sheaves with respect to this action with $\cR$-module structures by $\Sh_\heartsuit^{\bC}(M \times\bC_t)$. We denote the derived category by $\Sh^{\bC}(M \times\bC_t)$.

\subsection{Operations and monoidal structure}
We first recall some basic operations for equivariant sheaves. For details, we refer to \cite{BernsteinLunts, Tohoku}. Let $f\colon M\rightarrow N$ be a holomorphic map. We denote the induced map $M\times \bC_t\rightarrow M\times \bC_t$ by the same notation. Associated to $f$, we have functors
\begin{equation}
\begin{split}
    f_*, f_!&\colon \Sh^\bC(M\times \bC_t)\rightarrow \Sh^\bC(N\times \bC_t)\\
    f^{-1}, f^{!}&\colon \Sh^\bC(N\times \bC_t)\rightarrow  \Sh^\bC(M\times \bC_t).
    \end{split}
\end{equation}

Let us prepare some terminologies. Let $G$ be a group. Let $X_1, X_2$ be $G$-spaces. Let $f\colon X_1\rightarrow X_2$ be a $G$-map. Then we have functors
\begin{equation}
\begin{split}
    f_*, f_!&\colon \Sh^G(X_1)\rightarrow \Sh^G(X_2),\\
    f^{-1}, f^!&\colon \Sh^G(X_2)\rightarrow \Sh^G(X_1).
\end{split}
\end{equation}
For these functors, usual adjunctions hold. We can also define tensor and internal hom.

Let $H$ be another group with a surjective homomorphism $\phi\colon G\rightarrow H$. Let $Y$ be an $H$-space. Then $G$ acts on $H$ through $\phi$. Let $K$ be the kernel of $G\rightarrow H$. Then we have the invariant functor
\begin{equation}
    (-)^K\colon \Sh^G(Y)\rightarrow \Sh^H(Y)
\end{equation}
and the coinvariant functor
\begin{equation}
    (-)_K\colon \Sh^G(Y)\rightarrow \Sh^H(Y).
\end{equation}

Also, if one has an $H$-equivariant sheaf on $Y$, it can also be considered as a $G$-equivariant sheaf:
\begin{equation}
    \iota^\phi\colon \Sh^H(Y)\rightarrow \Sh^G(Y).
\end{equation}
\begin{lemma}
$\iota^\phi$ is the right adjoint of $(-)_K$ and the left adjoint of $(-)^K$.
\end{lemma}
\begin{proof}
For objects $\cE\in \Sh^G(Y)$ and $\cF\in \Sh^H(Y)$, the right factor of $\Hom_G(\cE, \iota_H^G\cF)$ is trivial with respect to $K$-action. Hence it reduces to $\Hom_H((\cE)_K, \cF)$.

The left factor of $\Hom_G(\iota_H^G\cF,\cE)$ is trivial with respect to $K$-action. Hence the image should be in $\cE^K$.
\end{proof}

Now let's go back to our situation. We again view $\bC$ as a discrete group. 

Let $\bC$ act on $M\times \bC_t$ by the addition on the right factor. We also consider $M\times \bC_t^2$ on which $\bC^2$ acts by the direct product of the previous action. Through the addition $a\colon \bC^2\rightarrow \bC$, the group $\bC^2$ also acts on $M\times \bC_t$. The kernel of the map $a\colon\bC^2\rightarrow \bC$ is the anti-diagonal $\Delta_a:=\{(a, -a)\in \bC\times \bC\}$.

We consider the addition map $m\colon M\times \bC_t\times \bC_t\rightarrow M\times \bC_t$ on $\bC_t$-factors. We then have a functor
\begin{equation}
    m_!\colon \Sh_S^{\bC^2}(M\times \bC_t^2)\rightarrow\Sh^{\bC^2}_S(M\times \bC_t).
\end{equation}
We also have a functor
\begin{equation}
    (-)_{\Delta_a}\colon \Sh^{\bC^2}_S(M\times \bC_t)\rightarrow \Sh^{\bC}_S(M\times \bC_t).
\end{equation}
We set 
\begin{equation}
    m_!^{\Delta_a}:=(-)_{\Delta_a}\circ m_!\colon \Sh_S^{\bC^2}(M\times \bC_t^2)\rightarrow\Sh^{\bC}_S(M\times \bC_t).
\end{equation}
We have the right adjoint of $m_!^{\Delta_a}$:
\begin{equation}
    m^!:=m^!\circ \iota^{a}\colon\Sh^{\bC}_S(M\times \bC_t)\rightarrow\Sh_S^{\bC^2}(M\times \bC_t^2).
\end{equation}

Let $p_i\colon M\times \bC_t^2\rightarrow M\times \bC_t$ be the $i$-th projection. We also have the corresponding projection $q_i\colon \bC^2\rightarrow \bC$. We then have
\begin{equation}
   p_{i*}^{\ker q_i}:=(-)^{\ker q_i}\circ p_{i*}\colon  \Sh_S^{\bC^2}(M\times \bC^2_t)\rightarrow \Sh_S^{\bC}(M\times \bC_t)
\end{equation}
and 
\begin{equation}
   p_i^{-1}:=p_i^{-1}\circ \iota^{q_i} \colon  \Sh_S^{\bC}(M\times \bC_t)\rightarrow \Sh_S^{\bC^2}(M\times \bC^2_t).
\end{equation}

\begin{lemma}
$p_i^{-1}$ is the left adjoint of $p_{i*}^{\ker q_i}$.
\end{lemma}

We set, for objects $\cE, \cF\in \Sh^{\bC}(M\times \bC_t)$, 
\begin{equation}
    \cE\star\cF:=m_!^{\Delta_a}(p_1^{-1}\cE\otimes p_2^{-1}\cF).
\end{equation}
\begin{example}
We consider the case when $M$ is the singlet, $\cR=\bC$, and $S=\bC_{\Re\geq 0}$.

Let $\bC_{\Re t\geq a}$ be the constant sheaf supported on $\lc t\in \bC\relmid \Re t \geq a\rc$. The direct sum $\bigoplus_{a\in \bC}\bC_{\Re t\geq a}$ has a natural $\bC$-equivariant structure (cf. the next section).
\begin{equation}
    p_i^{-1}\lb \bigoplus_{a\in \bC}\bC_{\Re t\geq a}\rb=\bigoplus_{a\in \bC}\bC_{\Re t_i\geq a}.
\end{equation}
Hence the tensor product is
\begin{equation}
    p_1^{-1}\lb \bigoplus_{a\in \bC}\bC_{\Re t\geq a}\rb\otimes p_2^{-1}\lb\bigoplus_{a\in \bC}\bC_{\Re t\geq a}\rb\cong \bigoplus_{a,b\in \bC}\bC_{\Re t_1\geq a, \Re t_2\geq b}.
\end{equation}
Applying $m_!^{\Delta}$ to this, we obtain
\begin{equation}
    \lb\bigoplus_{c\in \bC}\bigoplus_{d\in \bC}\bC_{\Re t\geq d}\rb_{\Delta_a}\cong \bigoplus_{d\in \bC}\bC_{\Re t\geq d}.
\end{equation}
\end{example}

\section{Sheaf-theoretic realizations of Novikov rings}

\subsection{Novikov ring via equivariant sheaves}
In this section, we will explain how we can realize the Novikov rings in the previous section using sheaves. The pioneer of such an idea is Tamarkin~\cite{Tam}, and a realization using equivariant sheaves appeared in \cite{kuwagaki2020sheaf}.

We first review the construction in \cite{kuwagaki2020sheaf}. Let $\bR_t$ be the real line with the standard coordinate $t$. We consider another real line $\bR_a:=\bR$ as an additive group by the usual addition. We equip $\bR_a$ with the discrete topology. The group $\bR_a$ acts on $\bR_t$ continuously by the translation. 

Again, the content in this section works over any field, but we will use the complex numbers $\bC$. We will consider the $\bC$-valued equivariant sheaves on $\bR_t$ with respect to $\bR_a$-action.

As an example of such an equivariant sheaf, we consider $\bigoplus_{c\in \bR}\bC_{t\geq c}$, where $\bC_{t\geq c}$ is the constant sheaf supported on the closed set $\lc t\in \bR_t\relmid t\geq c \rc$. For $a\in \bR_a$, we denote the translation of $\bR_t$ by $T_a$. To equip $\bigoplus_{c\in \bR}\bC_{t\geq c}$ with an equivariant structure, we have to specify an isomorphism
\begin{equation}
     T^{-1}\lb\bigoplus_{c\in \bR}\bC_{t\geq c}\rb\cong p_2^{-1}\lb \bigoplus_{c\in \bR}\bC_{t\geq c}\rb
\end{equation}
where $T\colon \bR_t\times \bR_a\rightarrow \bR_t$ is defined by $(t,a)\mapsto t+a$ and $p_2\colon \bR_t\times \bR_a\rightarrow \bR_a$ is the second projection. Since $\bR_a$ is equipped with the discrete topology, it is enough to have
\begin{equation}
T_a^{-1}\lb\bigoplus_{c\in \bR}\bC_{t\geq c}\rb\cong \bigoplus_{c\in \bR}\bC_{t\geq c}.
\end{equation}
Note that we have a canonical isomorphism
\begin{equation}
    T_a^{-1}\bC_{t\geq c} \cong \bC_{t\geq c-a}.
\end{equation}
Then the desired isomorphism is obtained as a direct sum of this isomorphism over any $a\in \bR$.

We would like to evaluate the endomorphism ring of $\bigoplus_c\bC_{t\geq c}$ in the category of $\bR_a$-equivariant sheaves. As described in \cite{kuwagaki2020sheaf}, we obtain a canonical isomorphism
\begin{equation}
    \End\lb \bigoplus_c\bC_{t\geq c}\rb \cong \Lambda_0.
\end{equation}
The isomorphism is given by sending $T^a\in \Lambda_0$ for $a\geq 0$ to the direct sum of the morphisms
\begin{equation}
    \bC_{t\geq c}\rightarrow \bC_{t\geq c+a},
\end{equation}
which is the identity on $\lc t\geq c+a\rc$.

We consider a variant of this construction. We consider the complex plane $\bC_t$ with the standard coordinate $t$. As in the real setting, we also consider the discrete additive group $\bC_a$ of complex numbers. Again, $\bC_a$ acts on $\bC_t$ by the translation continuously.

We consider the following sheaf $\bigoplus_{t\in \bC}\bC_{\Re t\geq \Re c}$ where $\bC_{\Re t\geq \Re c}$ is the constant sheaf supported on the closed subset
\begin{equation}
    \lc t\in \bC_t\relmid \Re t\geq \Re c\rc.
\end{equation}
This is also naturally a $\bC_a$-equivariant sheaf. We can compute the endomorphism ring similarly, and the result is the complexified Novikov ring:
\begin{equation}
    \Lambda_{0}^\bC \cong \End\lb \bigoplus_{t\in \bC}\bC_{\Re t\geq \Re c}\rb.
\end{equation}

More generally, we can do a similar construction for a sectoroid $S$. Let us consider the sheaf $\bigoplus_{c\in \bC}\bC_{S^\vee+c}$ where $S^\vee+c$ is the translation of $S^\vee$ by $c$. This sheaf $\bigoplus_{c\in \bC}\bC_{S^\vee+c}$ is also naturally a $\bC_a$-equivariant sheaf. We can similarly compute the endomorphism ring as
\begin{equation}
    \Lambda_0^S\cong \End\lb \bigoplus_{c\in \bC}\bC_{S^\vee+c}\rb.
\end{equation}

\subsection{Positive microsupport}
We would like to construct a category of sheaves equipped with $\Lambda_0^S$-action. 

Let $S$ be a sectoroid and $M$ be a complex manifold. There exists a standard trivialization $T^*\bC_t\cong \bC_t\times \bC$. We denote the standard cotangent coordinate by $\tau$. Also, through this trivialization, we implicitly consider the identification $\bC_\tau \cong \bC_{\hbar}$.

\begin{remark}
It is sometimes a question that we should consider whether $\tau=\hbar$ or $\tau=\hbar^{-1}$. It depends on how one considers the scaling of the cotangent bundle. Here we will take a point of view of $\tau=\hbar$.
\end{remark}

When we speak about the topological side, we will use the notation $\cOseL_S$ for the global section of $\cOseL_S$. From now on, we will use the notation $\Sh^\bC(M\times \bC_t)$ for the base ring $\cR=\cOseL_S$.

We denote the dual cotangent coordinate of $t$ by $\tau$. We set
\begin{equation}
    \cC_{S}:=\lc \cE\in \Sh^{\bC}(M \times\bC_t)\relmid \SS(\cE)\subset T^*M\times  \bC_t\times \overline{h(\Cone(S))^c}\rc
\end{equation}
is a thick subcategory of $\Sh^{\bC}(M\times S \times\bC_t)$ where $\Cone(S)^c$ is the complement of $\Cone(S)$. We set
\begin{equation}
  \Sh^\bC_{S}(M \times\bC_t):=\Sh^{\bC}(M \times\bC_t)/ \cC_{S}
\end{equation}
We can define the non-equivariant version $\Sh_{S}(M\times S \times\bC_t)$ similarly. Then the forgetful functor induces
\begin{equation}
    F\colon\Sh^\bC_{S}(M \times\bC_t)\rightarrow \Sh_{S}(M \times\bC_t).
\end{equation}

\begin{remark}
The nonequivariant version is thoroughly studied by Guillermou--Schapira~\cite{GS} based on \cite{KS, Tam}.
\end{remark}
\begin{remark}
There is a notion of relative constructible sheaves by \cite{FernandesSabbahMixed}. Our setting is very close to theirs.
\end{remark}

\subsection{Global microlocal cutoff and the Novikov ring action}
Since categorical quotients are usually difficult to work with, we will use a standard strategy of cutoff. We will keep the situation in the last section.

Let $c$ be a complex number.
For an object $\cE\in \Sh^{\bC}(M\times \bC_t)$, we consider the functor $\Sh^{\bC}(M\times \bC_t)\rightarrow \Sh^{\bC}(M\times  \bC_t); \cE\mapsto \cE\star\lb \bigoplus_{c\in \bC}\bC_{S^\vee+c}\rb$.

\begin{lemma}
The functor $(\bullet)\star \lb \bigoplus_{c\in \bC}\bC_{S^\vee+c}\rb$ is zero on $\cC_{S}$.
\end{lemma}
\begin{proof}
This is a corollary of \cite[Proposition 3.19]{GuiS}.
\end{proof}
Hence the functor induces another functor which will be denoted by the same notation:
\begin{equation}
    (\bullet)\star  \lb \bigoplus_{c\in \bC}\bC_{S^\vee+c}\rb\colon \Sh^{\bC}_{S}(M \times\bC_t)\rightarrow\Sh^{\bC}(M \times\bC_t).
\end{equation}
\begin{remark}[Non-equivariant cutoff]
Using the usual push-forward, we can define the non-equivariant version of the convolution product $\star_\noneq$ (which is the usual version ~\cite{GuiS}). Then the functor $(-)\star\bC_{t\geq 0}$ induces an endofunctor on the category of equivariant sheaves and is isomorphic to the equivariant cutoff $(-)\star \lb \bigoplus_{c\in \bC}\bC_{S^\vee+c}\rb$. 
\end{remark}

\begin{lemma}
The above functor is fully faithful. The left inverse is given by the quotient functor.
\end{lemma}
\begin{proof}
This is a corollary of \cite[Proposition 3.21]{GuiS}.
\end{proof}
In the following, we will tacitly identify $\Sh^{\bC}_{S}(M\times \bC_t)$ and its image under the functor $(\bullet)\star \lb \bigoplus_{c\in \bC}\bC_{S^\vee+c}\rb$. Working with the image, we can treat an object of $\Sh^{\bC}_{S}(M\times \bC_t)$ as an equivariant sheaf.

As a warm-up of this identification, we would like to state the above lemma in the following fashion:
\begin{lemma}
The endofunctor $(\bullet)\star \lb \bigoplus_{c\in \bC}\bC_{S^\vee+c}\rb$ on $\Sh^{\bC}_{S}(M \times \bC_t)$ is naturally isomorphic to $\id$. 
\end{lemma}

In view of this lemma, any object in $\Sh^\bC_S(M\times \bC_t)$ can be represented in a form that $\cE\star \lb \bigoplus_{c\in \bC}\bC_{S^\vee+c}\rb$. Since $\End\lb \bigoplus_{c\in \bC}\bC_{S^\vee+c}\rb\cong \Lambda_0^S$, we obtain the following proposition:
\begin{proposition}
There exists a canonical morphism $\Lambda_0^S\rightarrow \End(\cE)$ for any $\cE\in \Sh_S^\bC(M\times \bC_t)$.
\end{proposition}

\subsection{Novikov $\cHom$}
We would like to introduce Novikov $\cHom$, which is an analogue of $\cHom^E$ appeared in the theory of enhanced ind-sheaves~\cite{DK}.

Take $\cE, \cF\in \Sh^{\bC}_S(M\times \bC_t)$. Let $p_M\colon M\times \bC_t\rightarrow M$ be the first projection. Then $p_{M*}\cHom(F(\cE), F(\cF))$ is equipped with a diagonal $\bC$-action. We set
\begin{equation}
    \cHom^{\Lambda_0}(\cE, \cF):=\lb p_{M*}\cHom(F(\cE), F(\cF))\rb^{\bC}.
\end{equation}
et $\Lambda_{0M}$ be the constant sheaf of the Novikov ring on $M$. Then $ \cHom^{\Lambda_0}(\cE, \cF)$ is a module over $\Lambda_{0M}$.

\subsection{Module structure I}
We are considering the $\cOseL_S$-modules in the category $\Sh^{\bC}_{S}(M\times \bC_t)$. Hence we have a morphism
\begin{equation}
    \cOseL_S\rightarrow \cEnd^{\Lambda_0}(\cE).
\end{equation}
In particular, we get a morphism
\begin{equation}
    \Lambda_0\rightarrow \cEnd^{\Lambda_0}(\cE).
\end{equation}

On the other hand, as we have seen in the above, we have a morphism 
\begin{equation}
    \Lambda_0\rightarrow \cEnd^{\Lambda_0}(\cE).
\end{equation}
coming from the equivariant structure.

\begin{definition}
We say an object $\cE$ in the category $\Sh^{\bC}_{S}(M\times \bC_t)$ is compatible if the above two morphisms are isomorphic.
\end{definition}
The compatible modules form a triangulated subcategory of $\Sh_S^\bC(M\times \bC_t)$. We denote it by $\Sh_{\cO_S}^{\bC}(M\times \bC_t)$.

\subsection{Sheaf-theoretic subexponetial functions}

We consider $\cO^{\ae}_{\bC_t\times S}$ as a sheaf on $\bC_t$ by pushing-forward to $\bC_t$. This is canonically equipped with a $\bC_t$-equivariant structure. The differential operator $\hbar\partial_t-1$ acts on $\cO^{\ae}_{\bC_t\times S}$ equivariantly. We denote the kernel by $1_S$, which is also an equivariant sheaf.

\begin{lemma}
The ring $\cOseL_S$ acts on $1_S$.
\end{lemma}
\begin{proof}
Take an open subset $U$ in $\bC_t$. Note that any section $\psi$ of $1_S$ over $U$ has a form $\psi=\psi_0(\hbar)e^{t/\hbar}$. Take $\phi\in \cOseL_S$, then there $\phi e^{-c/\hbar}\in \cOaeL_S$ for any $c\in S^\vee\bs\{0\}$. For any $t\in U$, there exists $t'\in U$ and $c\in S^\vee\bs \{0\}$ such that $t=t'-c$. Hence we have
\begin{equation}
    \psi \phi =e^{t/\hbar}\psi_0\phi=e^{-c/2\hbar}\lb (e^{-c/2\hbar}\phi)\cdot e^{t'/\hbar}\psi_0\rb.
\end{equation}
Since $e^{t'/\hbar}\phi_0\in \cOaeL_S$ and $e^{-c/2\hbar}\psi$ is bounded, the function in the parentheses is bounded. Since $e^{-c/2\hbar}$ is rapid decay, $\psi\phi$ can be extended as a $C^\infty$-function on $U\times S$ (as a rapid decay function). This completes the proof.
\end{proof}
Hence we can consider $1_S$ as an object of $\Sh_S^\bC(\bC_t)$, but we do not go to this point of view for a while.

As a sheaf, $1_S$ admit the following description. Let $\psi$ be an element of $\cOexpL_S$. We have defined $\supp_\se(\psi)$ as a subset of $\bC_\hbar$. By the definition of the subexponential support, we have
\begin{equation}
    e^{t/\hbar}\psi\in \Gamma(\supp_{\se}(\psi), 1_S).
\end{equation}
By the isomorphism $\Gamma(\supp_{\se}(\psi), 1_S)\cong \Hom(\bC_{\supp_\se(\psi)}, 1_S)$, this specifies an inclusion map
\begin{equation}
    \bC_{\supp_\hbar(\psi)}\hookrightarrow 1_S.
\end{equation}
We denote the image of this morphism by the same notation. For a subset $F\subset \cO^{\Lambda}_S$, we can consider the subsheaf $\sum_{\psi\in F}\bC_{\supp_\se(\psi)}$. By the definition, we have
\begin{lemma}\label{lem: Lemma 7.11}
\begin{equation}
    \lim_{\substack{\longrightarrow \\ F\subset \cOexp_S}}\sum_{\psi\in F}\bC_{\supp_\se(\psi)}\cong 1_S.
\end{equation}
\end{lemma}

Using this presentation, we can write the action of $\cOseL_S$ on $1_S$ more explicitly. For $\psi\in \cOseL_S$ and $\phi\in \cOexpL_S$, we have remarked that $\supp_\se(\psi\phi)\supset \supp_\se(\phi)$. Hence we have a canonical nonzero morphism 
\begin{equation}
    \bC_{\supp_\se(\phi)}\rightarrow \bC_{\supp_\se(\psi\phi)}.
\end{equation}
These morphisms induce an endomorphism of $\displaystyle{\lim_{\substack{\longrightarrow \\ F\subset \cOexp_S}}}\sum_{\phi\in F}\bC_{\supp_\se(\phi)}$.

Now let us consider $1_S$ as an object of $\Sh_S^\bC(\bC_t)$. Only in this section, to emphasize the situation, we write it as $1_S\star \bC_{S^\vee}$. 
\begin{lemma}
The object $1_S\star \bC_{S^\vee}$ is compatible.
\end{lemma}
\begin{proof}
The equivariant structure of $1_S$ is given by the following: For an open subset $U\subset \bC$ and $t\in \bC$, we have an isomorphism
\begin{equation}
    1_S(U)\rightarrow 1_S(U+c);\psi\mapsto \psi e^{-c/\hbar}.
\end{equation}
Since the action of $\cOseL_S$ is the multiplication, the action commutes with the equivariant structure. From this, the action of $T^c$ is identified with $e^{-c/\hbar}$, which implies the compatibility.
\end{proof}
Hence we can consider $1_S$ as an object of $\Sh_{\cO_S}^\bC(\bC_t)$. The following is the main claim in this section:
\begin{proposition}\label{prop: endocalculation}
There exists a ring inclusion
\begin{equation}
\cOseL_S\hookrightarrow H^0(\End(1_S\star\bC_{S^\vee})). 
\end{equation}
\end{proposition}
\begin{proof}
We have already seen that there exists a homomorphism $\cOseL_S\rightarrow H^0(\End(1_S\star \bC_{S^\vee}))$. We will see this is injective.  

The each element $\psi\in \cOseL_S$ gives a homomorphism $\bC_{S^\vee}\rightarrow \bC_{\supp_\se(\psi)}\star \bC_{S^\vee}$. It is enough to see this is nonzero. The morphism is the image of the morphism $\bC_{-\Int(S^\vee)}\hookrightarrow \bC_{\supp_\se(\psi)}$ under the functor $(-)\star \bC_{S^\vee}$. Since $\psi$ is not very rapid decay, there exists a point $c_0\in S^\vee$ such that $c_0$ is not contained in $\supp_\se(\psi)$. Consider the complement $(c_0+S^\vee)^c$ of $c_0+S^\vee$. This is an open subset containing $\supp_\se(\psi)$. We now have the inclusions 
\begin{equation}
    \bC_{-\Int(S^\vee)}\hookrightarrow \bC_{\supp_\se(\psi)}\hookrightarrow\bC_{(c_0+S^\vee)^c}.
\end{equation}
Applying $\star \bC_{S^\vee}$ to this sequence, we obtain 
\begin{equation}
    \bC_{S^\vee}\rightarrow \bC_{\supp_\se(\psi)}\star \bC_{S^\vee}\rightarrow\bC_{c_0+S^\vee}
\end{equation}
and the composition is the canonical morphism given by the inclusion of the closed subsets $c_0+S^\vee\subset S^\vee$. Hence the first morphism in the sequence is nonzero. This completes the proof.
\end{proof}

We set
\begin{equation}
    \cO^{\mathrm{de}, \Lambda_0}_S:= \End(1_S\star \bC_{S^\vee}).
\end{equation}
The above theorem implies that this is a ring over $\cOseL_S$:
\begin{equation}
    \cOseL_S\rightarrow \cO^{\mathrm{de}, \Lambda_0}_S\cong \End(1_S\star \bC_{S^\vee}).
\end{equation}

\subsection{Module structure II}
Now we would like to slightly modify the formulation of the module structures. The modified definition is as follows: Instead of considering $\cOseL_{S}$-modules, we will consider $\cOdeL_S$-modules. From now on $\Sh_S^\bC(M\times \bC_t)$ represents the $\cOdeL_S$-version.

\begin{definition}
We say an object $\cE$ in the category $\Sh^{\bC}_{S}(M\times \bC_t)$ is compatible if the two morphisms $\Lambda_0\rightarrow \cEnd(\cE)$ are the same.
\end{definition}

\section{Sheaf quantization and Non-conic microsupport}
We introduce non-conic microsupport following \cite{Tam}. In the following, we assume that $M$ is a complex manifold.

\subsection{Over a point}

We denote the minimal triangulated subcategory of $\Sh^\bC_S(\{*\}\times  \bC_t)$ containing $1_S$ by $\SQ_S(T^*\{*\})$. By the definition of $\cOdeL_S$, we have:
\begin{corollary}\label{cor: overapoint}
Let $\mathrm{Perf}(\cOdeL_S)$ be the category of perfect modules over $\cOdeL_S$. We then have
\begin{equation}
    \SQ_S(T^*\{*\})\cong \mathrm{Perf}(\cOdeL_S).
\end{equation}
\end{corollary}

\subsection{Sheaf quantization}
We first define the absolute case. For an object $\cE\in \Sh_S^{\bC}(M\times\bC_t)$, its microsupport is well-defined in $T^*M\times \bC_t\times \Int(h(\Cone(S)))$. Let $\rho\colon T^*M\times \bC_t\times \Int(h(\Cone(S)))\rightarrow T^*M$ be the map defined by
\begin{equation}
    (x, \xi, t, \hbar)\mapsto (x, \xi/\hbar).
\end{equation}
We set
\begin{equation}
    \musupp(\cE):=\rho(\SS(\cE)\cap T^*M\times  \bC_t\times \Int(h(\Cone(S))).
\end{equation}

In this paper, we will only treat algebraic Lagrangians:
\begin{definition}
Let $L$ be a (possibly singular) holomorphic Lagrangian submanifold of $T^*M$. We say $L$ is algebraic if the local defining ideal is generated by functions in $\cO_M[\xi_1,...,\xi_n]$ where each $\xi_i$ is the cotangent coordinate.
\end{definition}
Now, we can define sheaf quantization.
\begin{definition}
An object $\cE\in \Sh_S^{\bC}(M\times\bC_t)$ is a sheaf quantization if 
\begin{enumerate}
    \item $\musupp(\cE)$ is an algebraic Lagrangian submanifold,
    \item the restriction $\cE_x$ is in $\SQ_S(T^*\{*\})$ for any point $x\in M$.
\end{enumerate}
\end{definition}
Note that a sheaf quantization is always compatible by the condition 2.

\begin{example}
Let $\alpha$ be a holomorphic function. We consider the sheaf
\begin{equation}
    \bigoplus_{c\in \bC}\bC_{S^\vee-\alpha(x)+c}.
\end{equation}
Note that the condition $t\in S^\vee-\alpha(x)$ is equivalent to
\begin{equation}
    \Re\lb \frac{t+\alpha}{\hbar}\rb \in \bR_{\geq 0}
\end{equation}
for $\hbar\in S$. This is equipped with an equivariant structure. We set
\begin{equation}
    S_\alpha:=\lb \bigoplus_{c\in \bC}\bC_{S^\vee-\alpha(x)+c} \rb\star 1_S.
\end{equation}
We claim that this is a sheaf quantization. Only the nontrivial to check is to compute $\musupp$.

For a constant $t\in \bC$,
the conormal of $t+\alpha(x)=0$ is spanned by $\tau+\frac{\partial \alpha}{\partial x_i}dx_i$. Let us denote the rays of $S$ by $R_1$ and $R_2$ and the positive orthogonal of them by $R^\perp_1$ and $R^\perp_2$. Then the microsupport is 
\begin{equation}
\begin{split}
    &\SS(\bC_{S^\vee-\alpha(x)})\\
    &=
    \lc \lb x, \tau \frac{\partial \alpha}{\partial x_i}, \alpha(x), \tau\rb\relmid \tau\in h(\Cone(S))\rc \cup\bigcup_{i=1,2} \lc \lb x, r_i \frac{\partial \alpha}{\partial x_i}, \alpha(x)+r_i^{\perp}, r_i \rb\relmid r_i\in R_i, r_i^\perp\in R_i^\perp \rc.
\end{split}
\end{equation}
Hence $\musupp\lb \bigoplus_{c\in \bC}\bC_{S^\vee-\alpha(x)+c} \rb=\Graph(d\alpha)$. Since $S_f$ is obtained as a certain colimit of $\bigoplus_{c\in \bC}\bC_{S^\vee-\alpha(x)+c}$, we also have $\musupp(S_\alpha)=\Graph(d\alpha)$. 
\end{example}

We denote the full subcategory spanned by sheaf quantizations by $\SQ_S(T^*M)$.
\begin{lemma}
The category $\SQ_S(T^*M)$ is closed under forming cones.
\end{lemma}
\begin{proof}
This is clear from the definition.
\end{proof}

\begin{remark}
Our formulation of sheaf quantization is slightly different from \cite{kuwagaki2020sheaf}. We explain how to translate them. 

Fix an $\hbar\in \bC_\hbar^\times$. We consider the half plane 
\begin{equation}
    S_{\hbar}:=\lc r\in \bC^\times \relmid \Re(r\overline{\hbar})\geq 0 \rc.
\end{equation}
Then the dual $S^\vee$ is a ray $\bR_{\geq 0}\cdot \hbar$. For a Lagrangian $L$, take a local primitive $\alpha$. The local model should be the one given in the above example:
\begin{equation}
    \bigoplus_{c\in \bC}\bC_{\bR_{\geq 0}\cdot \hbar-\alpha(x)+c}.
\end{equation}
We rotate $\hbar^{-1}$ by $\bC_\hbar$, the sheaf becomes
\begin{equation}
    \bigoplus_{c\in \bC}\bC_{\bR_{\geq 0}+(-\alpha(x)+c)/\hbar}
\end{equation}
This is now very close to the construction in \cite{kuwagaki2020sheaf} where we have defined 
\begin{equation}
    \bigoplus_{c\in \bR}\bC_{t\geq \Re(\alpha/\hbar)+c}
\end{equation}
in this situation. The only difference here is we are seeing how sheaves connected in the imaginary direction in $\bC_t$ which is responsible for the B-filed effect.
\end{remark}

\subsection{Stratification for sheaf quantization}
\begin{proposition}\label{prop: stratification}
Let $\cE$ be a sheaf quantization. Then there exists an analytic stratification $\cS=\{S_i\}$ of $M$ satisfying the following: For any $x\in S_i$, there exists a neighborhood $U\subset S_i$ of $x$ and a constructible sheaf $\cE_0$ defined on $U\times \bC_t$ and an object $\cE_1\in \SQ_S(T^*\{*\})$ such that $\cE|_{U}\cong (\bigoplus_cT_c\cE_0)\star  \cE_1$.
\end{proposition}
\begin{proof}
As presented in Appendix, take an adapted stratification $\cS$ of $L:=\musupp(\cE)$. For each $S\in \cS$, a standard calculation shows $\musupp(\cE|_S)$ is $\widetilde{L_S}/T^*_SM$ and it is a ramified covering of $S$. Hence we can assume that $\musupp(\cE)$ is a ramified covering from the outset. Further refining the stratification, we can assume that the covering is unramified covering. 

Let $U$ be a contractible open subset in $S_i$. Now the covering is trivial over $U$. We denote the primitives of the restriction of the Liouville form to the covering by $\{\alpha_i\}_i$ where the index set is the sheets of the covering. By a homotopy argument, $\cE_U$ can be classified by $\cE|_x$. Namely, we have the following description: We consider $\cE|_x$ as an object of $\SQ_S(T^*\{*\})$. Then we have
\begin{equation}
    \cE|_U\cong \lb \bigoplus_iS_{\alpha_i}\rb\star \cE|_x.
\end{equation}
This completes the proof.
\end{proof}

\subsection{Operations}
Let $f\colon M\rightarrow N$ be a holomorphic map. We have already introduced the functors
\begin{equation}
\begin{split}
    f_*, f_!&\colon \Sh^\bC(M\times \bC_t)\rightarrow \Sh^\bC(N\times \bC_t)\\
    f^{-1}, f^{!}&\colon \Sh^\bC(N\times \bC_t)\rightarrow  \Sh^\bC(M\times \bC_t).
    \end{split}
\end{equation}
Each of the functors induces a functor between $\Sh^\bC_S(M\times \bC_t)$ and $\Sh^\bC_S(N\times \bC_t)$. The followings are the direct consequences of the usual adjunctions.
\begin{proposition}
We have adjunctions:
\begin{equation}
    \begin{split}
        \Hom(\cF, f_*\cE)&\cong \Hom(f^{-1}\cF, \cE)\\
        \Hom(f_!\cE, \cF)&\cong \Hom(\cE, f^!\cF).
    \end{split}
\end{equation}
for $\cE\in \Sh^\bC_S(M\times \bC_t)$ and $\cF\in \Sh^\bC_S(N\times \bC_t)$. 
\end{proposition}

Applying to sheaf quantization, we obtain the following results.

\begin{lemma}
The functors $f^{-1}, f^!$ induce functors $\SQ_S(N)\rightarrow \SQ_S(M)$. 
\end{lemma}
\begin{proof}
For an object $\cE$ of $\SQ_S(N)$, by Proposition~\ref{prop: stratification}, there exists an analytic stratification of $N$ such that the restriction of $\cE$ to each stratum can be written as  $(\bigoplus_cT_c\cE_0)\star\cE_1$ for a constructible sheaf $\cE_0$ and an object $\cE_1\in \SQ_S(T^*\{*\})$. Then the $f^{-1}\cE$ is described as $\bigoplus_{c}T_cf^{-1}\cE_0\star \cE_1$ on the inverse image of the stratum. Hence $f^{-1}\cE\in \SQ(M)$. For $f^!$, similar proofs work.
\end{proof}

\begin{lemma}
If $f$ is proper, $f_*$ induces $\SQ_S(M)\rightarrow \SQ_S(N)$.
\end{lemma}
\begin{proof}
Again, for an object $\cE$ of $\SQ_S(M)$, we can again consider the strata decomposition. Using the notation in the previous proof, we have $f_*\cE\cong \bigoplus_{c}T_cf_*\cE_0\star \cE_1$. Hence this is also a sheaf quantization.
\end{proof}

\subsection{Monoidal structure}
We have introduced the convolution product $\star$. By replacing the tensor product in the definition of $\star$ with the tensor product over $\cOdeL_S$, we define $\star_\cO$. Via the equivalence Corollary~\ref{cor: overapoint}, this is the usual tensor product in $\mathrm{Perf}(\cOdeL_S)$. 

We next would like to construct the adjoint. 
\begin{equation}
    \cHom^{\star_\cO}(\cE, \cF):= p_{1*}\cHom(p_2^{-1}\cE, m^!\cF).
\end{equation}

\begin{lemma}
For $\cE, \cF, \cG\in \Sh^\bC_S(M\times \bC_t)$, we have
\begin{equation}
    \Hom(\cE\star_\cO \cF, \cG)\cong \Hom(\cE, \cHom^{\star_\cO}(\cF,\cG)).
\end{equation}
\end{lemma}
\begin{proof}
This is clear from the definition.
\end{proof}
From this, over a point, the adjoint is identified with usual internal hom in $\mathrm{Perf}(\cOdeL_S)$ via the equivalence Corollary~\ref{cor: overapoint}.

\begin{lemma}
For $\cE, \cF \in \Sh^\bC_S(M\times \bC_t)$, we have
\begin{equation}
    \Hom(\cE, \cF)\cong \Hom(1_S, \cHom^{\star_\cO}(\cE, \cF)).
\end{equation}
\end{lemma}
\begin{proof}
This is a corollary of the above lemma and $\cE\star_\cO 1_S\cong \cE$. 
\end{proof}

\begin{lemma}\label{Lemma: HomSQ}
For objects in $\cE, \cF\in \SQ(M)$, the objects $\cE\star_\cO\cF$ and $\cHom^{\star_\cO}(\cE, \cF)$ is again in $\mathrm{SQ}(M)$.
\end{lemma}
\begin{proof}
Again, by Proposition~\ref{prop: stratification}, we have a stratification such that we have presentations $\cE\cong (\bigoplus_{c\in \bC}T_c\cE_0)\star \cE_1$ and $\cF\cong (\bigoplus_cT_c\cF_0)\star \cF_1$ on each stratum. 
On the stratum, the object $\cE\star_\cO\cF$ is isomorphic to $(\bigoplus_{c\in \bC}T_c(\cE_0\star_{\noneq}\cF_0))\star (\cE_1\star_\cO \cF_1))\in \SQ_S(M)$. Inductively, one can show $\cE\star_\cO \cF\in \SQ_S(M)$. A similar discussion proves the statement for $\cHom^{\star_\cO}(\cE, \cF)$.
\end{proof}

\subsection{Duality}
Let $\omega$ be the orientation sheaf of $M$. 

We set
\begin{equation}
    \omega_S:=1_S\otimes \omega.
\end{equation}
We get the following duality functor
\begin{equation}
   \bD\cE:=\cHom^{\star_\cO}(\cE, \omega_S).
\end{equation}
By the definition and Lemma~\ref{Lemma: HomSQ}, $\bD\cE\in \SQ_S(M)$ for any $\cE\in \SQ_S(\cM)$.

\begin{lemma}
$\bD^2=\id$ on $\SQ_S(M)$. 
\end{lemma}
\begin{proof}
We note that the isomorphism
\begin{equation}
    \Hom(\cE, \bD^2\cE)\cong \Hom(\cE\star_\cO \cHom^{\star_\cO}(\cE, \omega_S), \omega_S)\cong \Hom(\cHom^{\star_\cO}(\cE, \omega_S), \cHom^{\star_\cO}(\cE, \omega_S)).
\end{equation}
By importing $\id$ on the right most side via this isomorphism, we get a morphism $\cE\rightarrow \bD^2\cE$. Again, it is enough to see on each stratification. Then this is a simple consequence of the perfectness appered in $\SQ_S(T^*\{0\})\cong \mathrm{Perf}(\cOdeL_S)$.
\end{proof}

We also would like to give an interpretation of the duality functor. Let $\delta\colon M\rightarrow M\times M$ be the diagonal embedding. We set
\begin{equation}
\begin{split}
    \bC_{\Delta}&:=\delta_*1_S\\
    \omega_{\Delta}&:=\delta_*(\omega_M\otimes 1_S).
\end{split}    
\end{equation}
where $\omega_M$ is the orientation sheaf.

For $\cE\in \SQ_S(M)$ and $\cF\in \SQ_S(N)$, we set $\cE\overset{\star_\cO}{\boxtimes}\cF:=p_1^{-1}\cE\star_\cO p_2^{-1}\cF\in \SQ_S(M\times N)$ where $p_i$ is the $i$-th projection.
We first remark the following isomorphism.
\begin{lemma}
\begin{equation}
    \begin{split}
        \Hom(\cE\overset{\star_\cO}{\boxtimes} \bC_\Delta, \omega_\Delta\overset{\star_\cO}{\boxtimes} \cF)&\cong\Hom(\cE, \cF)\\
        \Hom(\bC_\Delta\overset{\star_\cO}{\boxtimes} \cE, \cF\overset{\star_\cO}{\boxtimes} \omega_{\Delta})&\cong \Hom(\cE,\cF)
    \end{split}
\end{equation}
for $\cE, \cF\in \SQ_S(M)$.
\end{lemma}
\begin{proof}
One can import the proof of \cite[Lemma 9.4.4]{DK} to this setup.
\end{proof}

Similar to the differential equation case, we have the following criterion.
\begin{proposition}
For $\cE, \cF\in \SQ_S(M)$, we have $\cE\cong \bD\cF$ if and only if there exist two morphisms
\begin{equation}
    \begin{split}
 \cE\overset{\star}{\boxtimes}\cF&\xrightarrow{\epsilon}\omega_{\Delta}\\
\bC_{\Delta}\xrightarrow{\eta} \cF\overset{\star}{\boxtimes} \cE
    \end{split}
\end{equation}
such that the composition 
\begin{equation}
    \cE\overset{\star}{\boxtimes} \bC_{\Delta}\xrightarrow{\id\boxtimes \eta}\cE\overset{\star}{\boxtimes}\cF\overset{\star}{\boxtimes}\cE\xrightarrow{\epsilon\boxtimes\id}\omega_{\Delta}\overset{\star}{\boxtimes}\cE
\end{equation}
is identified with $\id$ in the above lemma and the composition
\begin{equation}
    \bC_\Delta\overset{\star}{\boxtimes}\cF\xrightarrow{\eta\boxtimes\id}\cF\overset{\star}{\boxtimes}\cE\overset{\star}{\boxtimes}\cE\xrightarrow{\epsilon\boxtimes\id}\omega_{\Delta}\overset{\star}{\boxtimes}\cE
\end{equation}
is identified with $\id$ in the above lemma.
\end{proposition}
\begin{proof}
This follows from the arguments of \cite[Section 9.4]{DK}.
\end{proof}

\begin{example}\label{ex: dualityexample}
Let $M$ be a contractible manifold. Let $f$ be a global holomorphic function. Then it is easy to see that $S_f\overset{\star}{\boxtimes} S_{-\alpha}\cong S_{\alpha}\overset{\star}{\boxtimes} S_f\cong \bC_{\Delta}$. This implies that $\bD S_f\cong S_{-f}$. If $M$ is not contractible, we have
$\bD S_\alpha\cong S_{-\alpha}\otimes \omega_M$. 
\end{example}

\section{Construction of the functor}

\subsection{The definition of the solution functor}
Let us fix a sectoroid $S$ and equip it with the sectoroid topology. Consider $\cOaeL_{M\times \bC_t\times S}$ as in Section 4. Consider the action of $\hbar\partial_t-1$ on it and take the kernel. We denote it by $\cR_{M\times \bC_t\times S}$. If $M$ is the singlet, this is $1_S$. Similar to the case of $1_S$, $\cR_{M\times \bC_t\times S}$ admits an action of $\cDdeL_{M\times S}$. The sheaf is equipped with a canonical $\bC_t$-equivariant structure.

\subsection{The definition of the solution functor}
The following construction is very much inspired from D'Agnolo--Kashiwara's construction~\cite{DK}.

We denote the quotient functor $\Sh^\bC(M\times \bC_t)\rightarrow \Sh_{S}^\bC(M\times \bC_t)$ by $[-]$. We get a functor
\begin{equation}
\begin{split}
    \Sol^\hbar_{S}\colon D^b(\cDdeL_{M\times S})&\rightarrow \Sh_{S}^\bC(M\times \bC_t) ;\\ \cM & \mapsto    p_{S*}\cHom_{\cDdeL_{M\times S}}(\cM,\cR_{M\times S\times \bC_t}).
    \end{split}
\end{equation}
where $p_S$ is the projection forgetting $S$.
This is our solution functor.

We will restrict the domain of the functor later. In particular, it will be locally contained in the image of 
\begin{equation}
    (-)\otimes_{\cOseL_S}\cOdeL_S\colon D^b(\cDseL_{M\times S})\rightarrow D^b(\cDdeL_{M\times S}).
\end{equation}
For an object $\cM\in D^b(\cDseL_{M\times S})$, we have
\begin{equation}
    \Sol^\hbar_{S}(\cM\otimes_{\cOseL_S}\cOdeL_S)\cong p_{S*}\cHom_{\cDseL_{M\times S}}(\cM,\cR_{M\times S\times \bC_t}).
\end{equation}
Hence we sometimes talk about the right hand side.

\subsection{De Rham functor}
As usual, we also introduce the de Rham functor. We first consider
\begin{equation}
     \Omega^{\cR}_{M\times  \bC_t\times S}:=\Omega_{M\times \bC_t}\otimes \cR_{M\times  \bC_t\times S}\
\end{equation}
This is canonically a right $\cDdeL_{M\times S}$-module.
We then get a functor
\begin{equation}
\begin{split}
    \DR_{S}\colon D^b(\cDdeL_{M\times S})&\rightarrow \Sh_{S}^\bC(M\times \bC_t) ;\\ \cM & \mapsto [p_{S*}(\Omega^{\cR}_{M\times  \bC_t\times S}\overset{\cD}{\otimes}\cM)].
    \end{split}
\end{equation}
Again, for an object of the form $\cM\otimes_{\cOseL_S}\cOdeL_S$, we have
\begin{equation}
    \DR_S(\cM\otimes_{\cOseL_S}\cOdeL_S)\cong [p_{S*}(\Omega^{\cR}_{M\times  \bC_t\times S}\overset{\cD}{\otimes}\cM)].
\end{equation}

Let $\bD$ be the Verdier dual functor.
\begin{proposition}[Verdier duality]
We have
\begin{equation}
    \DR_S\circ \bD\cong \Sol^\hbar_{S}.
\end{equation}
\end{proposition}
\begin{proof}
For a $\cDdeL_{M\times S}$-module, we have
\begin{equation}
\begin{split}
\Omega^{\cR}_{M\times \bC_t\times S}\overset{\cD}{\otimes}(\cHom_{\cD}(\cM, \cDdeL_{M\times S})\otimes (\Omega_M^{-1})^{\mathrm{de}, \Lambda_0})
&\cong \cHom_{\cD}(\cM, \cDdeL_{M\times S})\otimes_{\cO} \cR_{M\times \bC_t\times S}\\
&\cong \cHom(\cM, \cR_{M\times \bC_t\times S}).
\end{split}
\end{equation}
By pushing-forward along $p_S$ both sides, we complete the proof.
\end{proof}

\begin{proposition}[Duality 2]
\begin{equation}
    \bD\circ \DR_S\cong \DR_S\circ \bD.
\end{equation}
\end{proposition}
\begin{proof}
It is straight forward to follow the argument of the proof of \cite[Theorem 9.4.8]{DK}.
\end{proof}

By combining the above two results,
\begin{corollary}\label{cor: derhamduality}
We have
\begin{equation}
    \bD\circ \DR_S\cong \Sol^\hbar_{S}.
\end{equation}
\end{corollary}

\subsection{Example: 0-dimensional case}
We consider the case when $M$ is the singlet $\{*\}$. Then $\cOseL_S=\cDseL_S$. The rank 1 flat $\hbar$-connection in this category is $\cOseL_{S}$. Applying the solution functor
\begin{equation}
    \Sol^\hbar_{S}(\cO_S)=p_{S*}\cHom_{\cOseL_S}(\cOseL_S, \cR_{S\times \bC_t})\cong p_{S*}\cR_{S\times \bC_t}\cong 1_S
\end{equation}

\subsection{Example: 1st order}

Next consider the first order case.
The solutions for exponential $\cD^\hbar$-modules play important roles in the later discussion.

\begin{lemma}\label{lem: expsolutions}
Let $f$ be a holomorphic function on $M$. The solution for the exponential $\cDseL_{M\times S}$-module $\cE^{\alpha/\hbar}$ can be computed as
\begin{equation}
\begin{split}
    \Sol^\hbar_{S}(\cE^{\alpha/\hbar})&\cong S_\alpha\\
    \DR_S(\cE^{\alpha/\hbar})&\cong S_{-\alpha}
\end{split}
\end{equation}
\end{lemma}
\begin{proof}
By and Corollary~\ref{cor: derhamduality} and Example~\ref{ex: dualityexample}, it is enough to show one of the statements.
The solutions of $\cE^{\alpha/\hbar}\boxtimes \cE^{t/\hbar}$ in $\cO^{\ae}_{M\times S\times \bC_t}$ can be written as $e^{(f(x)+t)/\hbar}\psi(\hbar)$. For each $g\in \cO^{\Lambda}_S$ and a fixed $x$, the corresponding section $e^{\alpha(x)+t/\hbar}\psi(\hbar)$ gives a constant sheaf on $\supp_\se(\psi)-\alpha(x)$. From the description of Lemma~\ref{lem: Lemma 7.11}, the zeroth cohomology of the desired solution sheaf is isomorphic to $1_S\star_{\noneq} \bC_{S^\vee-\alpha(x)}\cong 1_S\star \bigoplus_{c\in \bC}\lb \bC_{S^\vee-\alpha(x)+c}\rb\cong S_\alpha$. 

To complete the proof, we have to show the higher cohomology vanishing. It is enough to show after the twisting by $\alpha$, namely, the case when $\alpha=0$. Also, it is enough to show that the higher cohomology vanishing for the de Rham complex. For a closed form $\phi$ valued in $\cR$, we would like to solve $\hbar\partial \psi=e^{t/\hbar}\phi$ by $\psi\in \cR$. Note that $e^{t/\hbar}g$ is rapid decay. By the usual proof of the holomorphic de Rham lemma, we can solve the above equation, by a rapid decay function, in particular, in $\cR$. This completes the proof.
\end{proof}

\subsection{Example: Local system}
Let us explain the correspondence on the level of local systems. We take $M$ to be the punctured disk $\bD^*$ in the complex line $\bC$. We denote the invertible elements of $\cOseL_S$ by $(\cOseL_S)^{\times}$.

We consider a sheaf quantization $\cL_\Psi$ which is locally isomorphic to $1_S\boxtimes \bC_M$ and the monodromy is given by $\Psi\in (\cOseL_S)^{\times}$. Note that $\musupp(\cL_M)$ is the zero section.
\begin{lemma}
For $\Psi\in (\cOseL_S)^{\times}$, we have
\begin{equation}
    \lim_{\hbar\rightarrow 0}\hbar\log(\Psi)=0.
\end{equation}
\end{lemma}
\begin{proof}
We first note that $\Psi e^{-c/\hbar}$ and $\Psi^{-1}e^{-c/\hbar}$ both go to zero as $\hbar\rightarrow 0$ for any $c\in S^\vee$. Hence
\begin{equation}
    |\Re \log \Psi(\hbar)|<\Re \frac{c}{\hbar}.
\end{equation}
for small $\hbar$.
We write $\hbar=r_\hbar e^{i\theta_\hbar}$. Then we have
\begin{equation}
    |\hbar \Re \log \Psi(\hbar)|=|r_\hbar \Re \log \Psi(\hbar)|<\Re \frac{c}{e^{i\theta_\hbar}}.
\end{equation}
Since $c$ can be taken arbitrary small, we conclude that $\lim_{\hbar\rightarrow 0}|\hbar| \Re \log \Psi(\hbar)=0$. 

Take a small compact open subset $U$ in $S$. We also take a sequence $\hbar_i=r_ie^{\sqrt{-1}\theta_i}\in S$ satisfying
\begin{enumerate}
\item $r_i$ is strictly decreasing and $r_i\rightarrow 0$ as $i\rightarrow \infty$,
\item $\theta_i$ is equal to some $\theta$ for any $i$, and 
\item $r_iU \subset S$.
\end{enumerate}
Then we consider a sequence of holomorphic functions $\psi_i$ on $r_0U$ by
\begin{equation}
   \psi_i(\hbar):=r_i\log \Psi|_{r_iU}\lb \frac{r_i}{r_0}\hbar\rb.
\end{equation}
The real part of this holomorphic function converges to zero by the above claim. In a small compact set, this is a uniform convergence. Hence $\psi_i$ also converges to zero. This completes the proof.
\end{proof}
By the above lemma, the $\hbar$-differential equation
\begin{equation}
    \lb \hbar d-\frac{\hbar}{2\pi \sqrt{-1}}\log \Psi\rb \psi=0
\end{equation}
is defined over $\cOceL_{M\times S}$ and the characteristic variety is the zero section. The solution set is spanned by $x^{\frac{\log \Psi}{2\pi\sqrt{-1}}}$, whose monodromy is $\Psi$.

\section{Localization}
\subsection{A remark}
We first remark that there is no hope to formulate $\hbar$-Riemann--Hilbert correspondence at the level of $\cDaeL$ or $\cDceL$ i.e., without inverting $\hbar$.

Consider the following two differential equations
\begin{equation}
\begin{split}
    \lb \hbar\partial +\frac{\hbar}{x}\rb \psi&=0,\\
    \lb \hbar\partial +\frac{2\hbar}{x}\rb \psi&=0.
\end{split}
\end{equation}
These two equations have the same characteristic variety $x\xi=0$. The first equation has a solution $\frac{1}{x}$ and the second one has a solution $\frac{2}{x}$. Hence both have the trivial monodromy, but they are not equivalent as $\cD^\hbar$-modules. 
This example tells us that $\hbar$-Riemann--Hilbert correspondence does not hold without inverting $\hbar$. 

Another issue is on the definition of localization in $\hbar$-setting, which is known in the context of Rees modules \cite{mochizukiMTM, MHMproject}.
We will consider on the complex plane $\bC$ and consider $0$ as a divisor. For an explanation, we will consider algebraically. We set
\begin{equation}
    \cO[*0]=\cD^\hbar_\bC\cdot (\bC[x, \hbar]\frac{1}{x})
\end{equation}
It is generated over $\bC[x,\hbar]$ by $\frac{1}{x}, \frac{\hbar}{x^2}, \frac{\hbar^2}{x^3}$ and so on. 

We consider the tensor product $\cO[*0]\otimes \cO[*0]$ over $\bC[x,\hbar]$ and the map 
\begin{equation}
    \cO[*0]\otimes \cO[*0]\rightarrow \cO[*0]; \psi\otimes \phi\mapsto \psi \phi x.
\end{equation}
This is a $\bC[x, \hbar]$-isomorphism. By this isomorphism, we induce $\cD^\hbar$-module structure on the RHS. Explicitly, it is defined by
\begin{equation}
    \hbar\partial (\psi):=\hbar\partial \psi+\frac{\hbar \psi}{x}.
\end{equation}
In other words, this is a flat $\hbar$-connection $\hbar\partial+\frac{\hbar}{x}$. This is not isomorphic to $\cO[*0]$. Hence it is not possible to define a reasonable notion of the localization using $\otimes \cO[*0]$. To remedy the situation, we should invert $\hbar$. After the inversion, the above two modules are isomorphic.
Note that $\cOseL$ and $\cOdeL$ contains $\hbar^{-1}$, hence the inversion is automatic.

\subsection{Localization modules}
We now explain a general setup.

We consider $\cOseL_{M\times S}$ as a $\cDseL_{M\times S}$-module. Let $D$ be a divisor in $M$. We can also consider $\cOseL_{M\times S}(*D)$ which admits poles on $D$, this is a coherent $\cDseL_{M\times S}$-module by the remark above.

For an object $\cM$ of $D^b_\hol(\cDseL_{M\times S})$, we set
\begin{equation}
    \cM(*D):=\cM\otimes \cO(*D)
\end{equation}
in $D^b_\hol(\cDseL_{M\times S})$.
The following is also standard.
\begin{lemma}\label{Lem: MD}
\begin{equation}
    \cHom(\cO(*D), \cM(*\cD))\cong \cM(*D).
\end{equation}
\end{lemma}

\section{Flat $\hbar$-connection and WKB-regularity}
We introduce the fundamental constituent of $\cD^\hbar$-modules.

\subsection{Flat $\hbar$-connection}
We would like to begin with a remark (which is not needed logically for our main purpose). Let $\cM$ be a coherent $\cDceL_{M\times S}$-module. We set $L:=\Char(\cM)$. By the classification of algebraic Lagrangians, there exists a divisor $D$ in $M$ such that $L\bs \pi^{-1}(D)\rightarrow M\bs D$ is a ramified covering. By pulling-back to each lower-dimensional stratum, we can again find a Zariski open subset on which the characteristic variety is a ramified covering. So the $\cDceL$-modules with ramified covering characteristic varieties are basic constituents of $\cDceL$-modules. So if one wants to impose some conditions to $\cDceL$-modules, it is enough to define $\cDceL$-modules with ramified covering characteristic varieties. This is our basic strategy to define the WKB-regularity and the summability.

For the later use, we would like to prepare the following definition.
\begin{definition}
We say $\cM$ is a meromorphic flat $\hbar$-connection if there exists a divisor $D$ such that 
\begin{enumerate}
    \item $\cM|_{M\bs D}$ is a flat $\hbar$-connection, and 
    \item $\cM\hookrightarrow \cM(*D)$ is an isomorphism $\cM\cong \cM(*D)$.
\end{enumerate}
\end{definition}

\begin{definition}
Given a $\cDceL_{M\times S}$-module $\cM$, we denote $\mathrm{Pole}(\cM)$ be the subset of $M$ where $\Char(\cM)\rightarrow M$ is not finite. 
\end{definition}

\subsection{WKB-regularity}
For flat $\hbar$-connections, we will introduce the notion of WKB-regularity. We would like to keep the notation in the previous section. 

Let $\cM$ be a flat $\hbar$-connection valued in  $\cOc_{M\times S}$. Then, for any $\hbar\in S$, we can define the corresponding (usual) flat connection by the specialization $\cM_\hbar$ at $\hbar$. Then the set of poles of $\cM_\hbar$ is contained in $\mathrm{Pole}(\cM)$.
For any point of $x$ in the pole divisor of $\cM_\hbar$, if $\cM_\hbar$ is holonomic, Sabbah--Mochizuki--Kedlaya theorem~\cite{SabbahIrreg, WildTwistor, Kedlaya} tells us that there exists a local modification $p\colon U'\rightarrow U$ on $U\ni x$ such that the pull back of $\cM_\hbar$ has a formal decomposition. Namely, there exists a set of multi-valued meromorphic functions $\{\alpha_i\}$ such that the formalization of the module is isomorphic to
\begin{equation}
    \bigoplus_{i\in I}\cE^{\alpha_i}\otimes \cR_i
\end{equation}
where each $\cR_i$ is a regular module and $\cE^{\alpha_i}$ is a connection $d-d\alpha_i$ (or its Galois-invariant version if $\alpha_i$ is multi-valued). Also, the graph of $d\alpha_i$ is an algebraic Lagrangian.

Since we are working with $\cDceL_{M\times S}$-module, we also include SMK theorem as an axiom, though we believe it holds for this setting as well.
\begin{definition}
Let $\cM$ be a meromorphic flat $\hbar$-connection. We set $L:=\Char(\cM)$.
We say $\cM$ is WKB-regular if the following holds: For any point of $x\in \mathrm{Pole}(\cM)$, there exists a local modification $p\colon U'\rightarrow U$ of a neighborhood $U\ni x$ such that
\begin{enumerate}
    \item the exceptional divisor is normal crossing type,
    \item after further pulling-back along a ramified map (to cancel the multi-valuedness), the formalization of $p^*\cM\otimes_{\cOceL_{M\times S}}\cOseL_{M\times S}$ at any point $y$ of the exceptional divisor of $p$ admit  a decomposition of the form
    \begin{equation}
    \bigoplus_{i\in I}\cM_i
\end{equation}
where each $\cE_i$ is a flat $\hbar$-connection, $\Char(\cM_i)\rightarrow M$ is a simple trivial covering on its image,
\item taking a primitive $\alpha_i$ of the Liouville form restricted to $\Char(\cM_i)$, we have an isomorphism
\begin{equation}
    \cM_i\otimes \cOseL_{M\times S}\cong \cE_i^{\alpha_i/\hbar}\otimes \cR_i
\end{equation}
where $\cR_i$ is a flat $\hbar$-connection with $\Char(\cR_i)\cap \pi^{-1}p^{-1}(M\bs D)\subset T^*_MM$ and $I_{\Char(\cR_i)}\cdot \cR_i/\hbar\cR_i=0$ where $I_{\mathrm{Char(\cR_i)}}$ is the ideal sheaf of $\Char(\cR_i)$. 
\end{enumerate}
\end{definition}
The meaning of the definition is that $\Char(\cM)$ represents the formal type. We would like to extend the definition to the case of $\cDceL$-modules.

\begin{definition}\label{def: WKB-regularD-mod}
An object $\cM$ of $D^b_{\hol}(\cDceL_{M\times S})$ is said to be WKB-regular if each cohomology module satisfy the following: there exists an analytic stratification $\cS$. For any $S\in \cS$, take a smoothing $p_S\colon S'\rightarrow \overline{S}\rightarrow M$. We suppose the localization $p_S^*\cM(S'\bs p_S^{-1}(S))$ is WKB-regular for any $S$.
\end{definition}

\section{Solving differential equations with Novikov coefficients}
We would like to define the summability of $\cD^\hbar$-modules in the next section. It is about solving $\hbar$-differential equation in the Novikov coefficient. We first would like to fix the notion of a solution in this setup.

\subsection{Linear}\label{subsection: linear}
We first would like to mention a reasonable definition of the solutions in the Novikov setting. Consider the following differential equation:
\begin{equation}
    \hbar d\psi+A\psi=0
\end{equation}
where $f$ is a complex-valued function and $A$ is a constant. We view this equation as a flat $\hbar$-connection valued in $\cOceL_{M\times S}$. It is natural to consider the solution of this equation is of the form $\exp(-\int A/\hbar)$, which is not in $\cOceL_{M\times S}$ in general, but in $\cO^{\Lambda}_{M\times S}$. 

\begin{definition}
Given a linear $\hbar$-differential equation of the form
\begin{equation}
    \hbar d\psi +A\psi =0
\end{equation}
where $\psi$ is a complex-vector-valued function and $A$ is a matrix valued in $\cOceL_{M\times S}$, a solution to this equation is a complex-vector-valued function $\psi$ with each entry in $\cO^{\Lambda}_{M\times S}$.
\end{definition}

\begin{theorem}\label{Theorem: Linear}
Let $\cE, \nabla$ be a rank 1 flat $\hbar$-connection valued in $\cOceL_{M\times S}$. Given an initial condition in $\cO^{\Lambda}_S$, we can solve it uniquely.
\end{theorem}
Since the connection is rank 1, its characteristic variety is locally a trivial covering of $M$. Take a primitive $\alpha$ of the restriction of the Liouville form to the characteristic variety. Then $\nabla=\hbar d+ d\alpha+A_0$ where $A_0|_{\hbar}=0$. We set $\widetilde{\psi}:=e^{\alpha/\hbar}\psi$, then we have
\begin{equation}
    \hbar d\widetilde{\psi}+A_{0}\widetilde{\psi}=0
\end{equation}
where $A_0$ is a function in $\cOceL_{M\times S}$ with $A_{0}|_{\hbar=0}=0$.
We will concentrate on this problem.

We first consider the 1-dimensional case. Since the question is local, we can assume the base manifold is just $\bC$. We denote $A_{0}$ in the form 
\begin{equation}
\begin{split}
    A_0&=B_0+B_1e^{-c_1/\hbar}+B_2e^{-c_2/\hbar}+\cdots=B_0+B_{>0}
\end{split}
\end{equation}
where each entry of $B_0$ and $B_i$ is in $\cOc_{\bC\times S}$ and the set $\{c_i\}$ is $\gamma$-finite for $\gamma=S^\vee$. Since $B_0|_{\hbar=0}=0$, we have $\exp(\int B_0/\hbar)\in \cOseL_{\bC\times S}$. We also note that $\exp(\int B_{>0}/\hbar)$ can be considered as an element of $\cOseL_{\bC\times S}$ by understanding it through the Taylor expansion. Hence we get a solution
\begin{equation}
\widetilde{\psi}=\exp\lb \int B_0/\hbar\rb \exp\lb \int B_{>0}/\hbar\rb \in \cOseL_{\bC\times S}.
\end{equation}

The uniqueness follows from a usual discussion as follows: Let $\widetilde{\psi}$ and $\widetilde{\psi}_1$ be solutions with $\widetilde{\psi}(0)=\widetilde{\psi}_1(0)=1$. Then $\widetilde{\psi}_1^{-1}$ solves 
\begin{equation}
\begin{cases}
        \hbar\frac{\partial \widetilde{\psi}_1^{-1}}{\partial z}&=-A_0\widetilde{\psi}_1^{-1}\\
        \widetilde{\psi}_1^{-1}(0)&=1.
\end{cases}
\end{equation}
This further implies that $(\widetilde{\psi}\widetilde{\psi}_1^{-1})'=0$. Evaluating at $0$, we can see that $\widetilde{\psi}\widetilde{\psi}_1^{-1}=1$. Hence $\widetilde{\psi}\widetilde{\psi}=\widetilde{\psi}\widetilde{\psi}_1$.

Now let's go to the higher dimensional case. We will prove by induction (we will mimic the discussion in \cite{SabbahIsomonodromic}).
Let $v$ be an initial condition. Let us suppose we have the statement for $\dim=n-1$. Consider the restricted problem to $z_1=0$. Associated to the initial condition, we can solve $n-1$-dimensional equation and get a solution $\sigma(0, z_2,...,z_{n-1})$. Then we next solve the 1-dimensional problems with the initial conditions $\sigma(0, z_2,...,z_{n-1})$. The resulting solution solves the original problem by the integrability, which is a unique solution. Note that we can carry out this procedure in $\cOseL_{M\times S}$.

Going back to the original problem, we get a solution of the form $\exp(\alpha/\hbar)\widetilde{\psi}$ where $\widetilde{\psi}\in \cOseL_{M\times S}$. 
An outcome of this consideration is that the $\cDseL_{M\times S}$-module defined by $\hbar d\psi +A\psi$ is isomorphic to $\cE^{\alpha/\hbar}$. We would like to state it as a corollary:

\begin{corollary}\label{cor: rank1summability}
Let $\cE, \nabla$ be a rank 1 flat $\hbar$-connection valued in $\cOceL_{M\times S}$ whose characteristic variety has a primitive $\alpha$. Then it is isomorphic to $\cE^{\alpha/\hbar}$ as a $\cDseL_{M\times S}$-module.
\end{corollary}

\subsection{Non-linear}
Of course, we do not treat general non-linear equations. Here we would like to present a Novikov version of \cite[Theorem 26.1]{Wasow}.

\begin{theorem}\label{Theorem: Non-linear}
Let $f(x, z, \hbar)$ be an element of $\cOaeL_{\bC\times \bC^n\times S}$. Consider the differential equation
\begin{equation}
    \hbar\partial_x\psi=f(x, \psi,\hbar).
\end{equation}
The formal version of this equation is
\begin{equation}
        \hbar\partial_x\psi=[f(x, \psi,\hbar)]
\end{equation}
where $[f]\in \cOaeL/T^{>0}\cOaeL\cong \cO_{\bC\times \bC^n}[[\hbar]]$. 
We assume that Jacobian of $[f]$ with respect to $z$-variables is nondegenerate at $\hbar\rightarrow 0$. Take a ray $\gamma$ of $S$. Suppose there exists a formal solution
\begin{equation}
    \sum_{i=1}^\infty \psi_i(x, z)\hbar^i
\end{equation}
to the differential equation. Then it admits a lift to a solution of the original equation in a neighborhood of $\gamma$.
\end{theorem}
\begin{proof}
The proof is a minor modification  of that in \cite{Wasow}. For the reader's convenience, we would like to provide the proof in detail.

We would like to decompose $f$ into the constant term, the linear term, and the higher term:
\begin{equation}
    f(x, z, \hbar)=a(x, \hbar)+A(x, \hbar)z+g(x, z, \hbar).
\end{equation}
Let $T$ be a linear transformation which makes $A(0,0)$ into a Jordan normal form with nonzero diagonals. Without loss of generality, we can a priori assume that $A(0,0)$ is such a Jordan matrix.

By the Borel--Ritt lemma, there exists an analytic function $\phi^*(x,z,\hbar)$ whose asymptotic expansion is $\sum_{i=1}^\infty \psi_i(x, z)\hbar^i$. By setting $u:=\psi-\psi^*$, we can take the differential equation into the form
\begin{equation}
    \hbar \partial_x u=b(x, \hbar)+B(x,\hbar)z+h(x, u, \hbar)
\end{equation}
where $b$ is asymptotically $0$ and $B(0,0)=J$. Now it is enough to prove the existence of the solution to the above equation with 0 asymptotics.

We next set
\begin{equation}
    V(x,\hbar)=\exp\lb\frac{Jx}{\hbar} \rb.
\end{equation}
We set 
\begin{equation}
    u(x, \hbar)=\int_{\Gamma(x)}\exp\lb \frac{(x-t)J}{\hbar}\rb\frac{b(t,\hbar)+(B(t,\hbar)-J)u+h(t, u,\hbar)}{\hbar}dt.
\end{equation}
Note that the right hand side of the equation should be understood as a power series in $e^{-c/\hbar}$. The integration path $\Gamma(x)$ is the same as in \cite[\S 27.3]{Wasow}. 

We now consider the operator
\begin{equation}
     \cP u(x, \hbar):=\int_{\Gamma(x)}\exp\lb \frac{(x-t)J}{\hbar}\rb\frac{b(t,\hbar)+(B(t,\hbar)-J)u+h(x, u,\hbar)}{\hbar}
\end{equation}
and discuss the convergence of the series
\begin{equation}
    \sum_{i=0}^\infty(\cP^{i+1}u_0-\cP^iu_0).
\end{equation}
for $u_0=0$. It is enough to show that the summation is convergent in $\Gr_D$ for any $D$. We can write $b, B, h$ as $b=b_D+b_{>D}, B=B_D+B_{>D}, h=h_D+h_{>D}$ where $b_D, B_D, h_D\in \cO^{\ae}$ and  $b_D, B_D, h_D$ are zero in $\Gr_D$. Let us set
\begin{equation}
    \begin{split}
    \cP_Du(t, \hbar)&:=\int_{\Gamma(x)}\exp\lb \frac{(x-t)J}{\hbar}\rb\frac{b_D(t,\hbar)+(B_D(t,\hbar)-J)u+h_D(t, u,\hbar)}{\hbar}\\
    \cP_{>D}u&:=\cP u-\cP_0u=\int_{\Gamma(x)}\exp\lb \frac{(x-t)J}{\hbar}\rb\frac{b_{>D}(t,\hbar)+B_{>D}(t,\hbar)u+h_{>D}(t, u,\hbar)}{\hbar}\\
\end{split}
\end{equation}
We claim that $\cP_{D}$ and $\cP_{>D}$ acts on $\Gr_D$. In fact, by \cite[Lemma 27.1]{Wasow}, there exists a constant $K$ independent of $\chi$ and $\hbar$ such that
\begin{equation}
    \left | \int_{\Gamma(x)}\exp\lb \frac{(x-t)J}{\hbar}\rb\frac{\chi(x)}{\hbar}\right|\leq K\max ||\chi||.
\end{equation}
Hence it preserves the asymptotics. In particular, $\cP_{>D}$ act as $0$ on $\Gr_D$. Hence it is enough to show that the convergence of 
\begin{equation}
    \sum_{i=0}^\infty(\cP_D^{i+1}u_0-\cP_D^iu_0).
\end{equation}
This can be proved exactly in the same way as in \cite{Wasow}. The resulting convergent series $\sum_{i=0}^\infty(\cP^{i+1}u_0-\cP^iu_0)$ gives the desired solution.
\end{proof}

\section{Summability}
Motivated by the study of exact WKB analysis and the definition of WKB-regularity, we would like to axiomatize the existence of solutions.

\subsection{Summability}
Let $S$ be a sectoroid. Let $\cM$ be an $\cOceL_{M\times S}$-valued flat $\hbar$-connection of $\rank =m$ defined on the sectoroid $S$. We set $D:=\mathrm{Pole}(\cM)$. As we have noted earlier, $L:=\Char(\cM)$ is a ramified covering of the complement of a divisor $D$ in $M$. We set
\begin{equation}
\begin{split}
    \mathrm{Ram}(\cM)&:=\lc x\in \Char(\cM)\relmid x \text{ is a ramified/singular point of }\Char (\cM)\rc\\
    \mathrm{Turn}(\cM)&:=p_{T^*M}(\mathrm{Ram}(\cM))\subset M\bs D.
\end{split}
\end{equation}

For a point $x\in M\bs (\mathrm{Turn}(\cM)\cup D)$, there exists a local neighborhood $U$ of $x$ such that $T^*U\cap L$ is a trivial smooth covering with $n$ sheets; we denote it by $\bigsqcup_{i=1}^n L_i$. Each sheet is equipped with the multiplicity $n_i$. Note that $\sum_{i=1}^nn_i=m$. Fix a primitive $\alpha_i$ of $L_i$ for each $i$.

\begin{definition}
For a point $x\in M\bs (\mathrm{Turn}(\cM)\cup D)$, we say $\cM$ is locally semisimple at $x$ if there exists a decomposition of $\cM\otimes_{\cOceL_{M\times S}}\otimes \cOseL_{M\times S}$ defined in a neighborhood $U$ of $x$ such that 
\begin{equation}
    \cM\otimes_{\cOceL_{M\times S}}\otimes \cOseL_{M\times S}|_U\cong \bigoplus_{i=1}^n(\cE^{\alpha_i/\hbar})^{\oplus n_i}.
\end{equation}
\end{definition}

\begin{remark}
By Corollary~\ref{cor: rank1summability}, to check the local semisimplicity, it is enough to show that $\cM\otimes_{\cOceL_{M\times S}}\otimes \cOseL_{M\times S}$ is decomposed into rank 1 flat connection coming from $\cOceL_{M\times S}$. 
\end{remark}

\begin{definition}
For a point $x\in D$, let us first take a local modification $U$ which makes the divisor into a normal crossing form. Then take the oriented real blow-up $\varpi\colon \widetilde{U}\rightarrow U$ of the divisor and consider the pull-back $\cM$. 
We say $\cM$ is locally semisimple at $x$ if, for any point $x\in \widetilde{U}$, there exists an open neighborhood $V$ of $x$ and a decomposition 
\begin{equation}
    ((\varpi^{-1}\cM\otimes_{\varpi^{-1}\cOceL_{U\times S}} \cO^{\mathrm{ce},\Lambda_0, \mathrm{mod}D}_{\widetilde{U}\times S})\otimes_{\cOceL_{M\times S}}\otimes \cOseL_{M\times S})|_V\cong \bigoplus_{i=1}^n(\cE^{\alpha_i/\hbar})^{\oplus n_i}.
\end{equation}

\end{definition}

\begin{definition}
We say $\cM$ is semi-globally semisimple if $\cM$ is locally semisimple at any point of $M$. 
\end{definition}

\begin{remark}
The semi-global summability implies the WKB-regularity.
\end{remark}

\begin{definition}
For a point $x\in M\bs (\mathrm{Turn}(\cM)\cup D)$, we say $\cM$ is locally summable at $x$ if there exists a basis of solutions $\lc \psi_{ij}\relmid 1\leq i\leq n, 1\leq j\leq n_i\rc$ defined in a neighborhood of $x$ such that $\psi_{ij}=e^{\alpha_i/\hbar}\widetilde{\psi}_{ij}$ is $\widetilde{\psi}_{ij}\in \cOseL_{U\times S}$ with the property that the matrix given by $\{\widetilde{\psi}_{ij}\}$ is an invertible matrix in $\cOseL_{U\times S}$.

We can similarly define the semi-global summability, which we call the summability.
\end{definition}

\begin{proposition}
The local semisimplicity and the local summability are equivalent.
\end{proposition}
\begin{proof}
The implication from the local semisimiplicity to the local summability is obvious: Each $\cE^{\alpha_i/\hbar}$ admit a solution $e^{\alpha_i/\hbar}$.

Let us show the converse: Suppose that $\cM$ is locally summable. The matrix with rows are $\widetilde{\psi_{ij}}$ is an invertible matrix valued in $\cOseL_{M\times S}$. Hence we can gauge-transform using the matrix and obtain the local semisimplicity.
\end{proof}

\begin{definition}
Let $\cM$ be a $\cDceL_{M\times S}$-module. We say $\cM$ is a summable flat $\hbar$-connection if there exists a divisor $D$ such that 
\begin{enumerate}
\item $\cM|_{M\bs D}$ is a summable flat $\hbar$-connection,
    \item the projection $\Char(\cM)\cap T^*(M\bs D)\rightarrow M$ is an uramified covering, and 
    \item $\cM\cong \cM(*D)$. 
\end{enumerate}
\end{definition}

Standing on this notion, we can generalize it to $D^b(\cDceL_{M\times S})$.
\begin{definition}
For $\cM\in D^b(\cDceL_{M\times S})$, we say it is summable if the following holds: There exists an analytic stratification such that the restriction of the cohomology modules to each stratum is summable in the same way as in Definition~\ref{def: WKB-regularD-mod}. We denote the subcategory of summable objects by $D^b_{\mathrm{sum}}(\cDceL_{M\times S})$. 
\end{definition}

Let $\cM$ be an object of $D^b(\cDseL_{M\times S})$. We say $\cM$ is locally of $\mathrm{ce}$ if there exists a neighborhood $U$ for any point $x\in M$ and $\cM_U\in D^b(\cDceL_{M\times S})$ such that $\cM|_U\cong \cM_U\otimes_{\cOceL_{U\times S}}\cOseL_{U\times S}$. In this situation, we say $\cM_U$ is a local $\ce$-lattice.

\begin{definition}
We say an object $\cM\in D^b(\cDseL_{M\times S})$ is summable if it is locally of $\mathrm{ce}$ and local $\ce$-lattices can be taken to be summable.

We denote the subcategory of $D^b(\cDseL_{M\times S})$ generated by the summable objects by $\cD_{\mathrm{sum}}(\cDseL_{M\times S})$.
Similarly, we define  $D^b_{\mathrm{sum}}(\cDdeL_{M\times S})$.
\end{definition}

We also have the germ version. Let $\theta\in S\cap \varpi^{-1}(0)$. 
\begin{definition}
We say $\cM\in D^b(\cDceL_{M\times \theta})$ is summable if the following holds: For any point $x\in M$, there exists a neighborhood $U$ of $x$ and a sectoroid $S$ containing $\theta$ such that $\cM|_U$ has a representative in $D^b(\cDceL_{U\times S})$ and the representative is summable.

We can similarly define the summability of an object of $\cM\in D^b(\cDseL_{M\times \theta})$.
\end{definition}

Later, we will provide a sufficient condition of the germ summability.

\subsection{Stability of summable modules}
In this section, we will prove that the localization and restriction preserves summability.

We start with a simple lemma.
\begin{lemma}
Let $\cM$ be a summable flat connection. Let $i\colon Z\hookrightarrow\cM$ be a subvariety. Then the restriction ${}_\cD i^*\cM$ is summable.
\end{lemma}
\begin{proof}
This is clear.
\end{proof}

\begin{lemma}
Let $\cM$ be a summable holonomic $\cD^\hbar$-module. Let $i\colon Z\hookrightarrow M$ be a subvariety of $M$. Then the restriction ${}_\cD i^*\cM$ is again summable.
\end{lemma}
\begin{proof}
Since $\cM$ is given as an extension of summable flat connections, the result follows from the preceding lemma.
\end{proof}

\begin{lemma}
Let $\cM$ be a summable flat connection. Then $\cM(*D)$ is also summable.
\end{lemma}
\begin{proof}
This is also easy.
\end{proof}

\subsection{Some properties}

\begin{proposition}[Pull-back]
Let $f\colon N\hookrightarrow M$ be a closed embedding of complex manifolds. 
For $\cM\in D^b_{\mathrm{sum}}(\cDseL_{M\times S})$, we have
\begin{equation}
    \Sol({}_\cD f^*\cM)\cong f^{-1}\Sol(\cM).
\end{equation}
\end{proposition}
\begin{proof}
It is enough to show for the summable flat $\hbar$-connections. Since locally $\cM$ is isomorphic to a direct sum of $\cE^{f/\hbar}$ and ${}_\cD f^*\cE^{f/\hbar}\cong \cE^{f|_N/\hbar}$, we complete the proof by Lemma~\ref{lem: expsolutions}.
\end{proof}

\begin{proposition}[Exterior tensor product]
For $\cM, \cN\in D^b_{\mathrm{sum}}(\cDseL_{M\times S})$, we have $\cM\boxtimes \cN\in D^b_{\mathrm{sum}}(\cDseL_{M\times M})$.
\end{proposition}
\begin{proof}
This is obvious.
\end{proof}

\begin{proposition}[Exterior tensor product]
For $\cM, \cN\in D^b_{\mathrm{sum}}(\cDseL_{M\times S})$, we have
\begin{equation}
    \Sol(\cM\boxtimes \cN)\cong \Sol(\cM)\overset{\star}{\boxtimes} \Sol(\cN).
\end{equation}
\end{proposition}
\begin{proof}
This is obvious for summable flat connections. Using the pull-back formulas, we can inductively prove the general case.
\end{proof}

\begin{proposition}[Tensor product]
For $\cM, \cN\in D^b_{\mathrm{sum}}(\cDseL_{M\times S})$, we have
\begin{equation}
    \Sol(\cM\otimes \cN)\cong \Sol(\cM)\overset{\star}{\otimes} \Sol(\cN).
\end{equation}
\end{proposition}
\begin{proof}
Consider the diagonal embedding $\Delta\colon M\rightarrow M\times M$ and apply the above pull-back formula to the previous lemma, we get the desired result. 
\end{proof}

For an object $\cM\in D^b(\cDseL_{M\times S})$, we set
\begin{equation}
    \cM(*D):=\cM\otimes \cO(*D).
\end{equation}
The following is a direct corollary.
\begin{corollary}\label{Lem: Tensor-restriction}
\begin{equation}
    \Sol(\cM(*D))\cong \Sol(\cM)\otimes \bC_{M\bs D}.
\end{equation}
\end{corollary}

\subsection{Solution sheaf of the summable connections}
Let $\cM$ be a $S$-summable flat $\hbar$-connection defined on $S$.

\begin{proposition}
The image of $\cM$ under the solution functor is a sheaf quantization.
\end{proposition}
\begin{proof}
Since sheaf quantization is stable under forming cones, it is enough to check for locally summable flat $\hbar$-connections. But, this is the case we have seen in Lemma~\ref{lem: expsolutions}.
\end{proof}

Although, we do not use the following theorem in this paper, it is useful to note that
\begin{conjecture}
For any $\cDceL$-module, we have
\begin{equation}
    \Char(\cM)=\mu supp(\Sol(\cM)).
\end{equation}
\end{conjecture}
It is easy to check this for summable flat $\hbar$-connections.

\section{The proof of the main theorem}
\subsection{The main theorem}
We would like to state our main theorem in this paper.

We have defined $D^b_{\mathrm{sum}}(\cDdeL_{M\times S})$ and $\Sol^\hbar_{S}$. We denote the induced functor by the same notation:
\begin{equation}
    \Sol^\hbar_{S}\colon D^b_{\mathrm{sum}}(\cDdeL_{M\times S})\rightarrow \SQ_S(T^*M).
\end{equation}

\begin{theorem}
The above functor $\Sol^\hbar_{S}$ induces an exact equivalence over $\cOdeL_S$.
\end{theorem}

\begin{remark}
Using the Novikov sheaf $\frakL_0$, one can state our theorem as an equivalence between $\frakL_0$-modules.
\end{remark}

\subsection{Reduction to flat connections}
We assume that the fully faithfulness is verified for flat connections with unramified characteristic varieties, which will be proved in the next section.

Let $\cM, \cN$ summable be $\cDdeL_{M\times S}$-modules. There exists a divisor $D$ such that $\cM(*D)$ and $\cN(*D)_{M\times S}$ are both summable flat connections.

Note that
\begin{equation}
    \Hom(\cM, \cN(*D))\cong \Hom(\cM, \cHom(\cO(*D),\cN(*D)))\cong \Hom(\cM(*D), \cN(*D)).
\end{equation}
by Lemma~\ref{Lem: MD}.
By the hypothesis, we have
\begin{equation}
    \Hom(\cM(*D), \cN(*D))\cong \Hom(\Sol^\hbar_{S}(\cN(*D)), \Sol^\hbar_{S}(\cM(*D))).
\end{equation}
By Lemma~\ref{Lem: Tensor-restriction},
\begin{equation}
\begin{split}
\Hom(\Sol^\hbar_{S}(\cN(*D)), \Sol^\hbar_{S}(\cM(*D)))&\cong \Hom(\Sol^\hbar_{S}(\cN)\otimes \bC_{M\bs D}, \Sol^\hbar_{S}(\cM)\otimes \bC_{M\bs D})\\
&\cong 
    \Hom(\Sol^\hbar_{S}(\cN(*D)), \Sol^\hbar_{S}(\cM))
\end{split}
\end{equation}
Hence we get
\begin{equation}
   \Hom(\cM, \cN(*D))\cong \Hom(\Sol^\hbar_{S}(\cN(*D)), \Sol^\hbar_{S}(\cM)).
\end{equation}

So it is enough to prove
\begin{equation}
    \Hom(\cM, \cC)\cong \Hom(\Sol^\hbar_{S}(\cC), \Sol^\hbar_{S}(\cM)).
\end{equation}
where $\cC:=\Cone(\cN\rightarrow \cN(*D))$ to see the desired isomorphism.

If $D$ is singular, we take a resolution $r_D\colon \widetilde{D}\rightarrow D$. We set $\widetilde{D}_1:=\mathrm{Pole}({}_\cD r_D^*\cC)$. We again have
\begin{equation}
     \Hom({}_\cD r_D^*\cM, {}_\cD r_D^*\cC(*\widetilde{D}_1))\cong  \Hom( \Sol^\hbar_{S}(r_D^*\cC(*\widetilde{D}_1)), \Sol^\hbar_{S}(r_D^*\cM))
\end{equation}
by the hypothesis. There exists a canonical morphism $\cC\rightarrow r_{D_*}r_D^*\cC(*\widetilde{D}_1)$ whose cone $\cC_1$ is supported on $r_D(\widetilde{D}_1)$. We have an exact triangle
\begin{equation}
    \Hom(\cM, \cC)\rightarrow  \Hom(r_D^*\cM, r_D^*\cC(*\widetilde{D}_1))\rightarrow \Hom(\cM, \cC_1)\xrightarrow{[1]}.
\end{equation}
Correspondingly,
\begin{equation}
    \Hom(\Sol^\hbar_{S}(\cC),\Sol^\hbar_{S}(\cM))\rightarrow \Hom(\Sol^\hbar_{S}({}_\cD r_D^*\cC(*\widetilde{D}_1)),\Sol^\hbar_{S}({}_\cD r_D^*\cM))\rightarrow \Hom(\Sol^\hbar_{S}(\cC_1), \Sol^\hbar_{S}(\cM))\xrightarrow{[1]}
\end{equation}
Hence it is enough to show
\begin{equation}
\Hom(\cM, \cC_1)\cong \Hom(\Sol^\hbar_{S}(\cC_1), \Sol^\hbar_{S}(\cM)).
\end{equation}
We can complete the proof by continuing the induction on the dimension of the support of $\cC_i$.

\subsection{Fully faithfulness}
Now we will prove the fully faithfulness.
\begin{proposition}
For summable meromorphic flat $\hbar$-connections $\cM, \cN$ (possibly with shifts) with poles in $D$, the morphism
\begin{equation}
        \Hom_{\cDseL_{M\times S}}(\cM, \cN)\rightarrow \Hom(\Sol^\hbar_{S}(\cN), \Sol^\hbar_{S}(\cM)).
\end{equation}
induced by $\Sol^\hbar_{S}$ is an isomorphism.
\end{proposition}
\begin{proof}
We will show locally. The strategy of the proof is close to the strategy in \cite{SabbahIntrotoStokes}.

We first blow-up the divisor $D$ locally to make it into a normal crossing divisor. Then we implement the oriented real blow-up $\varpi\colon \widetilde{U}\rightarrow U$. By the summability, in a small enough open subst of $\widetilde{U}$, $\cM$ and $\cN$ are isomorphic to direct sums of the modules of the form $\cE^{\alpha/\hbar}$ and $\cE^{\beta/\hbar}$, where $\alpha$ and $\beta$ are primitives of the characteristic varieties.

It is enough to show that 
\begin{equation}
    \cHom_\cD(\cE^{\alpha/\hbar}, \cE^{\beta/\hbar})\cong \cHom^{\Lambda_0}(\cS_\beta, \cS_\alpha).
\end{equation}
The left hand side is isomorphic to 
\begin{equation}
\begin{split}
    \cHom_\cD(\cE^{\alpha/\hbar}, \cE^{\beta/\hbar})&\cong\cHom_\cD(\cO^{\mathrm{de},\Lambda_0, \mathrm{mod} D}_{M\times S}, \cE^{(\beta-\alpha)/\hbar})\\
    &\cong \Cone(\cO^{\mathrm{de},\Lambda_0, \mathrm{mod} D}_{M\times S}\xrightarrow{\hbar d-d(\beta-\alpha)}\Omega^1_{M}\otimes_{\cO_M} \cO^{\mathrm{de},\Lambda_0, \mathrm{mod} D}_{M\times S}).
\end{split}
\end{equation}
We would like to compute the right hand side. We first note the following: Let $V$ be an open subset of $\widetilde{U}$. Then $S_\alpha|_V\cong \Cone(\cR^{\mathrm{mod}D}_{M\times \bC_t}\xrightarrow{\hbar d-d(\alpha)}\cR^{\mathrm{mod}D}_{M\times \bC_t})$ where $\cR^{\mathrm{mod}D}_{M\times \bC_t}:=\cR_{M\times \bC_t}\otimes_{\cOseL_{M\times S}} \cO^{\mathrm{se}, \Lambda_0, \mathrm{mod}D}_{M\times S}$. Actually, the solution is of the form $e^{(t+f)/\hbar}g(z,\hbar)$ where $g\in \cOseL_{M\times S}$ and the support is only where $t\in -S^\vee-f$. In this region, $e^{t+f/\hbar}g(z,\hbar)$ is of the moderate growth along $D$. Hence the right hand side is isomorphic to
\begin{equation}
\begin{split}
 \cHom^{\Lambda_0}(\cS_\beta, \cS_\alpha)&\cong (p_{M*} \cHom(1_S, \cS_{\alpha-\beta}))^\bC\\
&\cong \Cone((p_{M*}\cHom(1_S, \cR^{\mathrm{mod}D}_{M\times \bC_t}))^\bC\xrightarrow{\hbar d-d(\beta-\alpha)} (p_{M*}\cHom(1_S, \Omega^1_M\otimes_{\cO_M}\cR^{\mathrm{mod}D}_{M\times \bC_t}))^\bC)\\
 &\cong \Cone(\cO^{\mathrm{de},\Lambda_0, \mathrm{mod} D}_{M\times S}\xrightarrow{\hbar d-d(\beta-\alpha)}\Omega^1_{M}\otimes_{\cO_M} \cO^{\mathrm{de},\Lambda_0, \mathrm{mod} D}_{M\times S}).
\end{split}
\end{equation}
This completes the proof.
\end{proof}

\subsection{Essential surjectivity}
We consider the inverse functor as usual, but only for flat connections. We set
\begin{equation}
    \Psi(\cE):=\cHom^{\Lambda_0}(\cE, \cR_{M\times S\times \bC_t}).
\end{equation}
Note that this is a sheaf on $M$, but we consider it as a sheaf on $M\times S$ by tensoring $\cOdeL_{M\times S}$. We first check locally.

\begin{lemma}
For a flat $\hbar$-connection $\cE$, we have 
\begin{equation}
    \cE\cong     \Psi(\Sol^\hbar_{S}(\cE))\otimes_{\cOseL_{M\times S}}\cOdeL_{M\times S}.
\end{equation}
\end{lemma}
\begin{proof}
Note that there exists a canonical morphism $\cE\rightarrow \Psi(\Sol^\hbar_{S}(\cE))$. 
The sheaf quantization $\Sol^\hbar_{S}(\cE)$ is locally a direct sum of sheaf quantizations of the form $S_\alpha=(\bigoplus_{c\in \bC}\bC_{t\in c-\alpha+S^\vee})\star 1_S$ for a holomorphic function $\alpha$. 

Then we have
\begin{equation}
\begin{split}
    \Psi(\cS_\alpha)\otimes_{\cOseL_{M\times S}}\cOdeL_{M\times S}&\cong \cHom^{\Lambda_0}((\bigoplus_{c\in \bC}\bC_{t\in c-\alpha+S^\vee})\star 1_S, \cR)\otimes_{\cOseL_{M\times S}}\cOdeL_{M\times S}\\
    &\cong e^{\alpha/\hbar}\cOdeL_{M\times S}\\
    &\cong \cOdeL_{M\times S}.
\end{split}    
\end{equation}
Under the last isomorphism, the differential $\hbar d$ becomes $\hbar d+d\alpha$. This gives the desired identification.
\end{proof}

\begin{definition}
We say an object $\cE$ of $\SQ_S(T^*M)$ is elementary if 
\begin{enumerate}
    \item $\cE$ is supported on a smooth locally closed subset $N$ of $M$,
    \item the microsupport is stable under the addition of $T^*_NM$, and 
    \item the quotient $\musupp(\cE)/T^*_NM$ is an unramified covering of $N$.
\end{enumerate}
\end{definition}

\begin{proposition}
Any sheaf quantization can be obtained as an iterated cone of elementary sheaf quantizations.
\end{proposition}
\begin{proof}
This is the content of the proof of Proposition~\ref{prop: stratification}. 
\end{proof}

Hence it is enough to see that the solution functor hit all the elementary sheaf quantizations. This can be done in a way similar to the classical case. 

Let $\cS$ be an elementary sheaf quantization. Let $N$ be the support of $\cS$ in $M$. Take a desingularization $\widetilde{N}$ of the closure of $N$. We write the exceptional divisor by $E$. Then $\cS$ gives a sheaf quantization on $\widetilde{N}$ supported on $\widetilde{N}\bs E$. Consider $\Psi_S(\cS)$, which is a $\cDseL_{\widetilde{N}\bs E\times S}$-module. As an $\cOseL_{\widetilde{N}\bs E\times S}$-module, it is isomorphic to $(\cOseL_{\widetilde{N}\bs E})^{\oplus n}$ for some $n$. Inside this module, we can take $\cOseL_{\widetilde{N}}(*E)^{\oplus n}$. 

\begin{lemma}
The submodule is invariant under the action of $\cDseL_{\widetilde{N}\times S}$.
\end{lemma}
\begin{proof}
On each sector, the differential acts as a matrix of  $\hbar\partial-\alpha_i$. Since $\alpha_i$ is a meromorphic function with poles in $E$. Hence it acts on $\cOseL_{\widetilde{N}}(*E)^{\oplus n}$. 
\end{proof}

By the construction, this is our desired connection.

\begin{remark}
There is a slight difference between the concept of the monodromy in our setting and usual setting. Consider the following equation:
\begin{equation}
    \lb \hbar\partial_x -\frac{a}{x}\rb\psi=0.
\end{equation}
A solution to this equation is $x^{a/\hbar}$. Around $x=0$, it has monodromy $e^{2\pi \sqrt{-1}a/\hbar}$. If one fix a $\hbar$ and consider it as a usual differential equation, one has to treat it as monodromy. In particular, we have to treat it as a part of data of the corresponding Stokes local system. However, in our setting, it is not something to be encoded in Stokes data. Rather, it is encoded in the "formal type" of the equation. In other words, we have nontrivial formal types for regular $\hbar$-differential equations.
\end{remark}

\section{A sufficient condition for the summability}
In this section, we will prove a sufficient condition for the summability.

\begin{theorem}
Let $\cM$ be a weakly semi-globally semisimple WKB-regular flat $\hbar$-connection $\cM$ valued in $\cOaeL_{M\times S}$ defined over $S$. Then $\cM$ is summable at any $\theta\in S\cap \varpi^{-1}(0)$. 
\end{theorem}
The weak semi-glabal semisimplicity will be explained later.
As a corollary, we obtain the following.
\begin{corollary}
For any $\theta$, there exists a non-full faithful embedding
\begin{equation}
    D^b_{ws}(\cDae_{M\times \theta})\rightarrow \SQ_\theta(T^*M).
\end{equation}
where $D^b_{ws}(\cDae_{M\times S})$ is the subcategory generated by weakly semi-globally semisimple WKB-regular modules.
\end{corollary}

We follow \cite{SabbahIntrotoStokes} which is based on Mozhizuki's argument~\cite{WildTwistor}. Throughout this section, we fix a sectoroid $S$ and $\theta\in \varpi^{-1}(0)\cap S$. Since the variable $\hbar$ values in the ray spanned by $\theta$, it is convenient to rotate $\bC_\hbar$ so that $\theta=0$ and $\hbar$ values in the positive real numbers. This is our setting in this section.

\subsection{Weak diagonalization}
As explained in Wasow~\cite[\S 29]{Wasow}, it seems that the summability does not hold in general. The main problem is the possibility of local diagonalization of connection matrices. Hence we require it as a property.

\begin{definition}
\begin{enumerate}
    \item Let $\Omega$ be a connection matrix valued in $\cOaeL_{M\times S}$. We say it is diagonalizable if there exists an invertible matrix $P$ and $c\in \bR_{>0}$ such that the transformed $\Omega$ is a diagonal matrix.
    \item Let $\Omega$ be a connection matrix valued in $\cOaeL_{M\times S}$. We say it is weakly diagonalizable if there locally exists an invertible matrix $P$ and $c\in \bR_{>0}$ such that the transformed $\Omega$ has a form $\Omega_0+e^{-c/\hbar}\Omega_c$ where $\Omega_0, \Omega_c$ valued $\cOaeL_{M\times S}$ and $\Omega_0$ is a diagonal matrix.
\end{enumerate}
\end{definition}

If a given connection is diagonalizable, it is apparently summable. In the following, we would like to seek a sufficient condition to be diagonalizable.

\begin{proposition}
Suppose that $\Omega$ is a weakly diagonalizable connection. Then $\Omega$ is a diagonalizable connection.
\end{proposition}
\begin{proof}
Suppose $\Omega$ is weakly diagonalizable. Then $\Omega$ is in a form that
\begin{equation}
    \Omega=\Omega_0+e^{-c/\hbar}\Omega_{1}
\end{equation}
where $\Omega_0$ is a diagonal matrix. Take a primitive of each diagonal entry and denote it by $\alpha_i$.

Take point $x\in M$ and we will argue locally around $x$. Then 
\begin{equation}
    (e^{-c/\hbar}\Omega_c)^n e^{-\alpha_i/\hbar}=O(e^{-cn/\hbar}).
\end{equation}
Hence $\sum_{n\geq 0} (e^{-c/\hbar}\Omega_c)^n e^{-\alpha_i/\hbar}$ is convergent. This forms a desired basis of the solutions.
\end{proof}
Hence it is enough to seek some sufficient conditions for the weak diagonalizability.

\subsection{Weak semisimplicity}
In the following two sections, we introduce Sibuya's block-diagonalizatoin. Upon this, we will introduce the weak local semisimiplicity.

Let us fix $\theta\in S$. We will argue the reduction of higher rank to lower rank case. The setting is as follows: Let $\cE$ be a locally free $\cOaeL_{M\times S}$-module of finite rank and $\nabla$ be an $\hbar$-connection on it. As before, we assume that the characteristic variety it an unramified covering of $M\bs \mathrm{Pole}(\cM)$ from the beginning.

We will work locally in the following sense: For a point $x\in M\bs \mathrm{Pole(\cM)}$, we just take an open neighborhood. For a point $x\in \mathrm{Pole(\cM)}$, we will treat later.
Since we are free from turning points and working locally, we have $m$ sheets trivial covering of the characteristic variety over us. Let $\alpha_1,..,\alpha_m$ be primitives of these sheets. 

Since we will work locally, we can fix an isomorphism $\cE\cong (\cOaeL_{M\times S})^{\oplus M}$ for some $M$. Then the connection can be written as
\begin{equation}
    \nabla=\hbar d+\Omega=\hbar d+\sum_{i=1}^n \Omega^{(i)}dz^i
\end{equation}
where $\Omega^{(i)}$ is an $M\times M$-matrix valued in $\cOaeL_{M\times S}$.

We would like to construct a basis change $P$ which makes $\nabla$ into a block-diagonal form. This is done for 1-dimensional case by Sibuya~\cite{Sibuyadiagonal}. We first consider a formal construction. 

We first would like to find $Q\in \GL_M(\cO_{M}[[\hbar]])$ such that 
\begin{equation}\label{eq: eq14.4}
    Q^{-1}\hbar d_1Q+Q^{-1}\widehat{\Omega}^{(1)} Q=B
\end{equation}
where $B$ is a block-diagonal matrix. Here $\widehat{\Omega}^{(1)}$ is the asymptotic expansion of $\Omega^{(1)}$.

We rewrite this in the form
\begin{equation}
    \hbar Q'=\widehat{\Omega}^{(1)} Q-QB.
\end{equation}
where $'$ means $d_1$. We set
\begin{equation}
\begin{split}
    \widehat{\Omega}^{(1)}&=C_0(z)+\sum_{i=1}^\infty C_i(z)\hbar^i=:C_0+\widehat{C},\\     Q&=Q_0+\sum_{i=1}^\infty Q_i(z)\hbar^i=:Q_0+\widehat{Q},\\
    B&=B_0+\sum_{i=1}^\infty B_i(z)\hbar^i=:B_0+\widehat{B},\\
\end{split}
\end{equation}
Then the equation (\ref{eq: eq14.4}) can be read as
\begin{equation}
\begin{split}
    C_0Q_0-Q_0B_0&=0  \\
    C_0Q_r-Q_rC_0&=\sum_{s=0}^{r-1} (Q_sB_{r-s}-C_{r-s}Q_s)+Q'_{r-1}
\end{split}
\end{equation}

The following holds by the definition of the characteristic variety:
\begin{lemma}
The set of eigenvalues of $C_0$ coincides with $\lc d_1\alpha_1(x),..., d_1\alpha_n(x)\rc$.
\end{lemma}
We consider the case when $x\in M\bs \mathrm{Pole}(\cM)$.
By appropriately choosing $z_1$-direction, we can
suppose $d_1\alpha_i(0)\neq d_1\alpha_j(0)$ for $d\alpha_i\neq d\alpha_j$.
By applying a transformation initially, we can assume that $C_0(0)$ is in the Jordan normal form. Then the following holds:
\begin{lemma}[{\cite[Theorem 25.1]{Wasow}}]
We can make $C_0$ into a block-diagonal form where each block is labeled by eigenvalues.
\end{lemma}
Applying this lemma initially, We set $B_0=C_0$ and $Q_0=\id$.

Now we have
\begin{equation}
    C_0Q_r-Q_rC_0=C_r+\sum_{s=1}^{r-1} (Q_sB_{r-s}-C_{r-s}Q_s)+Q'_{r-1}
\end{equation}
Because $C_0$ is already a block diagonal, we can solve this in the form that $Q$ is anti-diagonal and $B$ is block-diagonal, inductively by a usual method as in \cite{Sibuyadiagonal, Wasow}.

Now we obtain $Q$. We redefine $\widehat{\Omega}$ by 
\begin{equation}
-Q^{-1}\hbar dQ+Q^{-1}\widehat{\Omega}Q
\end{equation}
and again set $\widehat{\Omega}=\sum_{i=1}^n\widehat{\Omega}^{(i)}dz^i$. Then the integrability condition tells us that
\begin{equation}
    \hbar\partial_1\widehat{\Omega}^{(j)}_{kl}=\widehat{\Omega}^{(j)}_{kl}\widehat{\Omega}_{ll}^{(1)}-\widehat{\Omega}^{(1)}_{kk}\widehat{\Omega}^{(j)}_{kl}
\end{equation}
where $k\neq l$ are block labelings. 

\begin{lemma}
Each $\widehat{\Omega}^{(i)}$ is block diagonal i.e., $\widehat{\Omega}^{(j)}_{kl}=0$ for $k\neq l$.
\end{lemma}
\begin{proof}
We write
\begin{equation}
    \widehat{\Omega}^{(j)}_{kl}=\sum_{i=0}^\infty C_{i,kl}^{(j)}(z)\hbar^i
\end{equation}
Let us first see the $0$-th term with respect to $\hbar$. We have
\begin{equation}
   0=C_{0,kl}^{(j)}C_{0, ll}-C_{0, kk}C^{(j)}_{0,kl}
\end{equation}

Since the eigenvalues of $C_{0, kk}$ and $C_{0, ll}$ are different on a dense open subset, $C^{(j)}_{0, kl}$ vanishes. We can now inductively see all the terms vanish. For example, the 1-st term is 
\begin{equation}
    \hbar\partial_1 C^{(j)}_{0, kl}=C_{1,kl}^{(j)}C_{0, ll}-C_{1, kk}C^{(j)}_{0,kl}+C_{0,kl}^{(j)}C_{1, ll}-C_{0, kka}C^{(j)}_{1,kl}.
\end{equation}
This equation is reduced to 
\begin{equation}
0=C_{1,kl}^{(j)}C_{0, ll}-C_{0, kk}C^{(j)}_{1,kl}.
\end{equation}
Again, by the same reasoning, we have $C^{(j)}_{1,kl}=0$.
\end{proof}

Hence we get a formal block diagonalization. Now $B$ is block diagonal. This constraint gives
\begin{equation}
\begin{split}
    \hbar\frac{d\widehat{Q}_{kl}}{dz_1}&=-C_{0klkk}\widehat{Q}_{kl}-\widehat{C}_{kl}-\sum \whC_{ks}\widehat{Q}_{sl}+\widehat{Q}_{kl}\wB_{ll}+\widehat{Q}_{kl}\whC_{0ll}+\widehat{Q}_{kl}\whC_{ll}\\
    \wB_{ll}&=\whC_{ll}+\sum \whC_{ls}\wQ_{sl}.
\end{split}
\end{equation}
where $\alpha,\beta$ are the block labelings. Combining these two equations, we get
\begin{equation}
    \hbar\frac{d\wQ_{kl}}{dz_1}=-C_{0kk}\wQ_{kl}-\whC_{kl}-\sum \whC_{ks}\wQ_{sl}+\wQ_{kl}(\whC_{ll}+\sum \whC_{ls}\wQ_{sl})+\wQ_{kl}C_{0ll}+\wQ_{kl}\whC_{ll}.
\end{equation}
Now we replace this with analytic coefficients:
\begin{equation}\label{eq: 14.16}
    \hbar\frac{dQ_{kl}}{dz_1}=-C_{0kk}Q_{kl}-C_{kl}-\sum C_{ks}Q_{sl}+Q_{kl}(C_{ll}+\sum C_{ls}Q_{sl})+Q_{kl}C_{0ll}+Q_{kl}C_{ll}
\end{equation}
where $C=\Omega^{(1)}-C_0$. By what we have seen, this admits a formal solution.

\begin{lemma}
The equation (\ref{eq: 14.16}) admits a nontrivial solution.
\end{lemma}
\begin{proof}
We first note that  we have see that (\ref{eq: 14.16}) admits a formal solution. The assumption in Theorem~\ref{Theorem: Non-linear} holds, since $k\neq l$ have different eigenvalues.
Hence the statement is a consequence of Theorem~\ref{Theorem: Non-linear}.
\end{proof}
Hence we can make $\Omega^{(1)}$ into a block-diagonal form in the case when $x\in M\bs \mathrm{Pole}(\cM)$.

For the case when $x\in \mathrm{Pole}(\cM)$, we consider a local modification into a normal crossing divisor and implement the oriented real blowup. Then take an open neighborhood in the blowed-up space. We fix a local coordinate $z_1,..., z_n$
in the neighborhood with $x=(0,0,\cdots, 0)$. We first make a modification so that the formal type becomes a good normal form. By the WKB-regularity, we have a formal decomposition $\bigoplus \cE^{\alpha_i/\hbar}$. Lifting it into analytic in the sector, we have an initial transform.  Now we can follow the first step of the proof of \cite[Proposition 12.8]{SabbahIntrotoStokes}. We can easily check that the obtained transformation is valued in $\cOseL_{M\times S}$ and make $\Omega^{(1)}$ into a block-diagonal form, also in this case. 

Now $\Omega$ is the block diagonal form where each block is labeled by a sheet of the characteristic variety.
We would like to see other components. By the integrability, we have
\begin{equation}
    \hbar\partial_1\Omega^{(j)}_{kl}=\Omega^{(j)}_{kl}\Omega_{ll}^{(1)}-\Omega^{(1)}_{kk}\Omega^{(j)}_{kl}.
\end{equation}
Hence $\Omega^{(j)}_{kl}$ must be of the form $e^{\alpha_k-\alpha_l+c_{kl}/\hbar}$ times an asymptotically expandable function for some $c_{kl}$. Since $\Omega^{(j)}_{kl}$ is asymptotically expandable and $\alpha_k\neq \alpha_l$, we have to have
\begin{equation}
   \Re( \alpha_k-\alpha_l+c_{kl}/\hbar)< 0.
\end{equation}
By taking an enoguh small compact neighborhood $U$, we can see that
\begin{equation}
    c_*:=\min\lc \Re(e^{-i\theta}(\alpha_k(x)-\alpha_l(x)-c_{kl}))\relmid k\neq k, x\in U\rc
\end{equation}
is a positive number. Hence we have a weak ``block-diagonalizability" i.e., $\Omega$ is a block-diagonal form up to $O(e^{-c_*/\hbar})$.
\begin{definition}
We say $\nabla$ is locally weakly semi-globally semisimple if there exists a further transformation that each block becomes diagonal up to $O(e^{-c_*/\hbar})$.
\end{definition}

For example, if the number of sheets of the characteristic variety coincides with the rank of the connection (i.e., each multiplicity is 1), then the connection is automatically weakly locally semisimple.

\section{Exact WKB analysis}

\subsection{Schr\"odinger case}
We briefly recall the summability result for second order $\hbar$-differential equations. For more complete explanations, please refer to \cite{Takei, IwakiNakanishi, kuwagaki2020sheaf}.

Let $C$ be a compact Riemann surface. Let $\cQ$ be a Schr\"odinger equation in the sense of \cite{IwakiNakanishi}. Let $D$ be the pole divisor of the equation. We fix a spin structure, then $\cQ$ defines a flat $\hbar$-connection. We assume the WKB-regularity.

Let $L$ be the characteristic variety, which is a Lagrangian. Let $x_0\in C$ be a turning point. An $\hbar$-Stokes curve emanating from $x_0$ is defined by
\begin{equation}
    \lc x\in C\relmid \Image\int_{x_0}^x\frac{\lambda|_L}{\hbar}=0\rc.
\end{equation}
Assuming the GMN condition~\cite{GMN, BridgelandSmith, IwakiNakanishi} and taking a generic $\hbar$, then the set of Stokes curves give a cell decomposition of $C$. We call each 2-dimensional cell a Stokes region. One of the fundamental theorem in exact WKB analysis is the following:
\begin{theorem}[Koike--Sch\"afke theorem]
Assume the above situation. For a small and generic $\hbar\in \bC^\times$, the exact WKB method works outside the Stokes curves.
\end{theorem}
For the proof, one can refer to \cite{Takei}. If one further assumes the simpleness of the turning points, the transformation occurring on the solution basis can be explicitly determined by Voros formula~\cite{Voros, KawaiTakei}.

Based on these results, in the previous paper~\cite{kuwagaki2020sheaf}, we construct a sheaf quantization from a Schr\"odinger equation as an $\bR$-equivariant sheaf. We explain a slightly modified construction which produces a $\bC$-equivariant sheaf. For the terminologies, we refer to loc. cit. 

Let us take a sectoroid $S$ such that the Koike--Sch\"afke theorem holds. We also fix $\hbar\in S$.

For each Stokes region $R$, take a base point $x_0\in R$. Over each Stokes region, the spectral curve is a trivial cover $L|_R=L_1\sqcup L_2$. We take a local sheaf quantization as
\begin{equation}
    \bigoplus_{c\in \bC}\bigoplus_{i=1,2}\bC_{\lc (x, t)\in C\times \bC_t\relmid t\in- \int_{x_0}^x  \frac{\lambda|_{L_i}}{\hbar}+c+S^\vee\rc} \star 1_S.
\end{equation}
This has canonically an $\bC$-equivariant structure.

We will glue up these local sheaf quantizations to produce a genuine sheaf quantization. In \cite{kuwagaki2020sheaf}, the gluing morphisms consist of the Voros gluings and the base point changes.

For the Voros gluing, we can implement exactly as in \cite{kuwagaki2020sheaf}. For the base point change, we will define it in a slightly different way. Let $x_1$ be another base point. Then the solutions are connected by the multiplication of 
\begin{equation}
    \exp\lb \int_{x_0}^{x_1} \frac{\lambda|_{L_i}}{\hbar}\rb\exp\lb\sum_{i\geq 0} S_i \rb.
\end{equation}
where the right part is an asymptotically expandable function.
We consider the associated identification of sheaves by taking the direct sum of 
\begin{equation}
\begin{split}
   \bC_{\lc (x, t)\in C\times \bC_t\relmid t\in- \int_{x_0}^x  \frac{\lambda|_{L_i}}{\hbar}+c+S^\vee\rc} \xrightarrow{\cdot \exp\lb\sum_{i\geq 0} S_i \rb} &\bC_{\lc (x, t)\in C\times \bC_t\relmid t\in- \int_{x_0}^x  \frac{\lambda|_{L_i}}{\hbar}+c+S^\vee\rc}\\
    &= \bC_{\lc (x, t)\in C\times \bC_t\relmid t\in- \int_{x_1}^x  \frac{\lambda|_{L_i}}{\hbar}+c'+S^\vee\rc}.
\end{split}
\end{equation}
where $c'=-\int_{x_0}^{x_1} \frac{\lambda|_{L_i}}{\hbar}+c$. We can also fulfill the turning points as in \cite{kuwagaki2020sheaf}. We would like to denote the resulting sheaf quantization by $S_\cQ$, which is defined on $C\bs D$. On the other hand, we can get a sheaf quantization by applying $\Sol^\hbar_{S}$ to $\cQ$ on $C\bs D$. Since both $S_{\cQ}$ and $\Sol^\hbar_{S}(\cQ)$ see the transformation between solutions of $\cQ$, we obtain the following:
\begin{proposition}
$S_{\cQ}\cong \Sol^\hbar_{S}(\cQ)$.
\end{proposition}

\begin{remark}
When we define the microlocalization in \cite{kuwagaki2020sheaf}, we include B-field twist in the definition. This is because we did not include the information about $\exp\lb \Im \int \frac{\lambda|_{L_i}}{\hbar}\rb$ in the construction of the sheaf quantizations. In the present construction, we have already included the information. So we do not need the twist. Instead, we slightly change the gluing construction of the local system on $L$ where we took only the information coming from $ \exp\lb\Re \int \frac{\lambda|_{L_i}}{\hbar}\rb$. By using the complexified Novikov field, it is obvious that we can also include the information coming from $\exp\lb\Im  \int \frac{\lambda|_{L_i}}{\hbar}\rb$. 
\end{remark}

\begin{remark}
Iwaki--Nakanishi~\cite{IwakiNakanishi} studied cluster transformations arising in the theory of exact WKB analysis. In \cite{kuwagaki2020sheaf}, it is interpreted as a different microlocalizations.

Here if one takes a sectoroid $S$ to be one contains $\hbar_1$ and $\hbar_2$, then we can take different microlocalizations: one is along the dual of $\hbar_1$ and one is along the dual of $\hbar_2$. These different microlocalizations are related by a cluster transformation if there exists a wall between $\hbar_1$ and $\hbar_2$.
\end{remark}

\subsection{Higher order case: Spectral scattering construction}
Let us sketch a higher order case. The details will be explained in elsewhere~\cite{KuwNov}. We will keep the same situation as in the last section, but we will allow higher order differential equations.

Let $\cQ$ be an $\hbar$-differential operator. By fixing a spin structure, we lift it to an $\hbar$-connection. We assume the WKB-regularity. In higher order case, the Stokes geometry is more complicated and the summability result is not available. Some discussions toward the summability can be found in \cite{Virtual}. We briefly discuss an ideal story (i.e., folklore expectations).

Let $L$ be the characteristic variety of $\cQ$. Again, we have the set of turning points. For each point $x\in C$ outside the turning points, we can construct a formal WKB solution. We apply the Borel transformation to this formal solutions. The first conjecture is the following:
\begin{conjecture}
The Borel transformation is endlessly continuable. 
\end{conjecture}
Roughly speaking, the endless continuability says that the function admits analytic continuation endlessly with certain poles. For a part of the precise definition, we refer to \cite{KamimotoSauzin}, while the original reference is Ecalle~\cite{Ecalle}. 

We further assume that 
\begin{conjecture}
The Borel transformation is Laplace transformable along any path going to $\infty$ avoiding poles.
\end{conjecture}
It is moreover expected that the Borel transformation is Ecalle--Laplace transformable for any reasonable averaging. See the discussion in \cite{EcalleM}.

To predict a distribution of the poles in the first sheet of the Riemann surface of the Borel transformation, we introduce the language of Stokes geometry.
\begin{definition}
Fix $\theta\in \bR/\bZ$. 
\begin{enumerate}
    \item A $\theta$-preStokes curve is a curve on $C$ on which $\int \frac{\lambda}{e^{i\theta}}\in \bR$ is satisfied.
    \item A 0-th Stokes curve is a $\theta$-preStokes curve emanating from a turning point.
    \item An $i$-th Stokes curve is a $\theta$-preStokes curve emanting from an intersection point of $j$-th and $j'$-th Stokes curves with $j, j'<i$.
\end{enumerate}
\end{definition}
As it is explained in \cite{kuwagaki2020sheaf}, each Stokes curve carries a weight (a.k.a. BPS mass). It is expected that we can construct a set $\frakS$ of Stokes curves such that each Stokes curve in it carrying transformations of the solution basis like
\begin{equation}
    \psi_i\rightsquigarrow \psi_i+\sum a_{ij}e^{-c/\hbar}\psi_j
\end{equation}
where $c$ is the weight. For a given point $x\in C$, the weights of the Stokes curves in $\frakS$ passing through $x$ is expected to be the set of poles of the Borel transformations. Also, $\frakS$ is expected to be finite at any cutoff at $c\in \bR_{>0}$. This property is called the Gromov compactness and studied in \cite{kuwagaki2020sheaf}.

If one can complete this story, we can construct a sheaf quantization from $\frakS$ as we have done in \cite[\S 12]{kuwagaki2020sheaf}. Again, since the resulting sheaf quantization sees the transformation matrices between the solution bases, it should coincide with the one obtained by $\Sol^\hbar_{S}$.

\begin{remark}
Assuming the conjectural full results of exact WKB analysis, one can partially refine our functor, for example, about coefficients. This will be discussed in elsewhere (e.g. \cite{KuwNov}).
\end{remark}

\section{Relation to the holonomic Riemann--Hilbert correspondence}
We would like to compare our sheaf quantization with D'Agnolo--Kashiwara's Riemann--Hilbert correspondence.

\subsection{Enhanced ind-sheaves and holonomic Riemann--Hilbert correspondence}
We would like to recall the theory of enhanced ind-sheaves briefly.

The topological side is the notion of enhanced ind-sheaves. We consider $\overline{\bR}:=[0,1]$ as a two point compactification of $\bR$ via $\bR\cong (0,1)\hookrightarrow [0,1]=\overline{\bR}$. We set 
\begin{equation}
D^b(\mathrm{I}\bC_{M\times(\overline{\bR},\bR)}):=D^b(\mathrm{I}\bC_{M\times\overline{\bR}})/D^b(\mathrm{I}\bC_{M\times{\overline{\bR}\bs \bR}})
\end{equation}
where $D^b(\mathrm{I}\bC_M)$ is the bounded derived category of ind-sheaves over $M$. 
We set $\bC_{t\lesseqgtr 0}:=\bC_{\lc (x,t)\in M\times \overline{\bR}\relmid t\in \bR, t\lesseqgtr 0 \rc}$. We denote the convolution product along $\bR_t$ by $\potimes$. We set 
\begin{equation}
\mathrm{IC}_{t^*=0}:=\lc K\relmid K\potimes \bC_{\leq 0}\simeq 0, K\potimes \bC_{\geq 0}\simeq 0\rc.
\end{equation}
The category of enhanced ind-sheaves over $M$ is defined by 
\begin{equation}
E^b({\mathrm{I}\bC_M}):=D^b(\mathrm{I}\bC_{M\times (\overline{\bR}, \bR)})/\mathrm{IC}_{t^*=0}.
\end{equation}
We set $\bC^E_M:=\underset{a \rightarrow\infty}{\forlim}\bC_{t\geq a}$ as an object of $E^b({\mathrm{I}\bC_M})$. As usual, $\forlim$ means Ind-colimit. An object of $E^b({\mathrm{I}\bC_M})$ is said to be $\bR$-constructible if there locally exists an $\bR$-constructible sheaf $\cE$ on $M\times \bR_t$ and the object is locally isomorphic to $\cE\potimes \bC^E_M$. 
\begin{theorem}[\cite{DK}]
There exists a contravariant fully faithful functor $\mathrm{Sol}^E\colon D^b_{hol}(\cD_M)\hookrightarrow E^b(\mathrm{I}\bC_M)$ where $D^b_{hol}(\cD_X)$ is the derived category of holonomic $\cD_M$-modules.
\end{theorem}

\subsection{Convergent lattice and associated enhanced ind-sheaf}

Fix a sectoroid $S$ and $\hbar\in S$.
Let $\cE$ be an object of $\SQ_S(T^*M)$. Suppose that there exists a divisor $D\subset M$ such that $\musupp(\cE)$ is on $M\bs D$ and the projection $\musupp(\cE)\rightarrow M\bs D$ is a ramified covering. For a point $x\in D$, take a local modification around $x$ which makes $D$ into a normal crossing divisor. Furthermore, take an oriented real blow up of the normal crossing divisor. We fix such a modification $\widetilde{M_x}$ for each point of $x\in D$.

\begin{definition}
In the above setting, a convergent lattice is the following data:
\begin{enumerate}
\item a collection $\lc U_i\rc$ of open coverings of $M\bs D$ and $\widetilde{M_x}$
    \item $\cE|_{U_i}\cong (\bigoplus_{c\in \bC}T_c\cE_i)\star 1_S$ where $\cE_i$ is a constructible sheaf in $\Sh_S(M\times \bC_t)$
\end{enumerate}
satisfying that
\begin{enumerate}
    \item $\cE_i\cong T_{c_{ij}}\cE_j$ on $U_i\cap U_j$ for some $c_{ij}\in \bC$,
    \item the gluing morphism $g_{ij}\colon \cE|_{U_i\cap U_j}\xrightarrow{\cong} \cE|_{U_i\cap U_j}$ is convergent in the following sense: The morphism space $\Hom(\cE|_{U_i\cap U_j}, \cE|_{U_i\cap U_j})$ is identified with ${\prod}^{S^\vee}_{c\in \bC}\Hom(\cE_i, T_c\cE_i)\otimes_{\Lambda_0^S}\otimes \cOseL_S$. There exists $D\in  \frakD$ such that the subspace of $\Hom(\cE|_{U_i\cap U_j}, \cE|_{U_i\cap U_j})$ defined as the image of the canonical map
    \begin{equation}
        \bigoplus_{c\in D}\Hom(\cE_i, T_c\cE_i)\otimes_{\bC}\cOse_S\rightarrow {\prod}^{S^\vee}_{c\in \bC}\Hom(\cE_i, T_c\cE_i)\otimes_{\Lambda_0^S}\otimes \cOseL_S
    \end{equation}
    contains $g_{ij}$.
\end{enumerate}
\end{definition}

We now associate an enhanced ind-sheaf to a convergent lattice. Let $\cC$ be a convergent lattice of $\cE$. Let $\{U_i\}$ be the defining open covering of the convergent structure. 

Let $\bR\cdot \hbar$ be the real line in $\bC_\hbar$ spanned by the fixed $\hbar\in S$. We set $\cE^\hbar:=\cE|_{\bR\cdot \hbar}$.

On each $U_i$, we have
\begin{equation}
    \cE^\hbar|_{U_i}\cong \bigoplus_c T_c\cE^\hbar_{i}\star 1_S.
\end{equation}
Specializing at $\hbar\in S$, on each overlap $U_i\cap U_j$, we get a morphism
\begin{equation}
   g_{ij}(\hbar)\colon \cE^\hbar_{i}|_{U_i\cap U_j}\rightarrow \bigoplus_{c\in D}T_c\cE^\hbar_{j}|_{U_i\cap U_j}.
\end{equation}

We denote each projection by
\begin{equation}
    g_{ij}(\hbar)_c\colon \cE^\hbar_{i}|_{U_i\cap U_j}\rightarrow T_c\cE^\hbar_{j}|_{U_i\cap U_j}.
\end{equation}
This further induces a morphism
\begin{equation}
     g_{ij}(\hbar)_c\colon \cE^\hbar_{i}\potimes \bC^E_M|_{U_i\cap U_j}\rightarrow  \cE^\hbar_{j}\potimes \bC^E_M|_{U_i\cap U_j}.
\end{equation}
We finally get a morphism
\begin{equation}
    \sum_{c\in D}e^{-c/\hbar} g_{ij}(\hbar)_c\colon \cE^\hbar_{i}\potimes \bC^E_M|_{U_i\cap U_j}\rightarrow  \cE^\hbar_{j}\potimes \bC^E_M|_{U_i\cap U_j}.
\end{equation}
Since the original morphisms satisfy the cocycle condition, these morphisms also satisfy the same cocycle condition. Hence this gives an enhanced ind-sheaf on the blowed-up space. By the push-forward along the projection, we get an enhanced ind-sheaf on $M$ supported on $M\bs D$. We denote the resulting enhanced ind-sheaf by $E(\cE, \cC)$.

\subsection{From sheaf quantization to enhanced ind-sheaf}
We now would like to explain the relation between D'Agnolo--Kashiwara's functor and our sheaf quantization.

Let $\cM$ be a flat connection valued in $\cOc_{M\times S}$. We slightly modify the definition of the summability in this situation.
\begin{definition}
For a point $x\in M\bs (\mathrm{Turn}(\cM)\cup D)$, we say $\cM$ is locally semisimple at $x$ if there exists a decomposition of $\cM\otimes_{\cOc_{M\times S}}\otimes \cOse_{M\times S}$ defined in a neighborhood $U$ of $x$ such that 
\begin{equation}
    \cM\otimes_{\cOc_{M\times S}}\otimes \cOse_{M\times S}|_U\cong \bigoplus_{i=1}^n(\cE^{\alpha_i/\hbar})^{\oplus n_i}.
\end{equation}
We can define the semi-global version and call it the summability.
\end{definition}
Let $\cM$ be a summable flat connection valued in $\cOc_{M\times S}$.
We can construct enhanced ind-sheaves in two ways. 

The first way is as follows: We set
\begin{equation}
   \cM_\hbar:= \cM\otimes_{\cOc_{M\times S}} \cO_M
\end{equation}
where $\cOc_{M\times S}\rightarrow \cO_M$ is given by the evaluation at $\hbar\in S$. Then the object canonically has a structure of $\cD_M$-module. Then we get an enhanced ind-sheaf $\Sol^E(\cM_\hbar)$.

The second way uses a convergent structure. By the summability assumption, we have a local isomorphism $ \cM\otimes_{\cOc_{M\times S}}\otimes \cOse_{M\times S}|_U\cong \bigoplus_{i=1}^n(\cE^{\alpha_i/\hbar})^{\oplus n_i}$. In particular, we get a basis of solutions of the form $e^{\alpha_i/\hbar}\psi_i$ where $\psi_i\in \cOse_{M\times S}$. Accordingly, we have a trivialization 
\begin{equation}
    \Sol^\hbar_{S}(\cM)|_U\cong \bigoplus_{i}\bigoplus_{c\in \bC}\bC_{S^\vee-\alpha_i+c}\star 1_S.
\end{equation}
Since the gluing morphisms come from  transformation matrices of the actual analytic solutions, this gives a convergent structure. We denote the convergent structure $C_\cM$. 
\begin{theorem}
There exists an isomorphism
\begin{equation}
    \Sol^E(\cM_\hbar)\cong E(\Sol^\hbar_{S}(\cM), C_\cM)
\end{equation}
\end{theorem}
\begin{proof}
Away from the divisor $D$, the data survive in enhanced ind-sheaves are the same as that of local systems on $M\bs D$. In the both constructions, the data come from the transition matrices of solutions, hence they are the same. On the divisor $D$, the data survive in the enhanced ind-sheaf is a Stokes structure. The formal type of the both constructions are the same, which is a direct consequence of the WKB-regularity. Again, the gluing come from the transition matrices of solutions, hence they are the same. This completes the proof.
\end{proof}

\begin{remark}[General $\cDae_{M\times S}$-module]
It is natural to ask a generalization of the above result to general summable $\cDae_{M\times S}$-module. 
However, it seems that the definition of convergent lattice does not work well for general summable $\cDae_{M\times S}$-module. A possible way to remedy this is as follows: First, decompose sheaf quantization to elementary sheaf quantizations. Then we can define the convergent lattice for each elementary sheaf quantization. Then we can associate an enhanced ind-sheaf for each. The extension map between elementary sheaf quantizations induce an extension map between enhanced ind-sheaf. The resulting object can be identified with the image under $\Sol^E$.
\end{remark}

\section{Appendix}
\subsection{Lagrangian dichotomy}
Let $M$ be a complex manifold. We equip the standard holomorphic symplectic structure $\omega$ with the holomorphic cotangent bundle $T^*M$. We denote the projection $T^*M\rightarrow M$ by $p_{T^*M}$.

\begin{definition}
Let $L$ be a (possibly singular) holomorphic Lagrangian submanifold of $T^*M$. We say $L$ is algebraic if the local defining ideal is generated by functions in $\cO_M[\xi_1,...,\xi_n]$ where each $\xi_i$ is the cotangent coordinate.
\end{definition}

\begin{lemma}
For a holomorphic $\cD^{\hbar, \al}_M$-module, its characteristic variety $\mathrm{Char}(M)$ is always an algebraic Lagrangian.
\end{lemma}
\begin{proof}
This is clear from the definition.
\end{proof}

For a moment, we consider $T^*M$ as a holomorphic vector bundle over $M$. Let $\cO$ be the trivial bundle. We set $\overline{T^*M}:=\mathrm{tot}(\bP(T^*M\oplus \cO))$, which is the total space of projective compactification of $T^*M$. Then $T^*M$ can be canonically embedded into $\overline{T^*M}$. For a holomorphic Lagrangian $L$ in $T^*M$, we denote the closure of $L$ in $\overline{T^*M}$ by $\overline{L}$.
\begin{lemma}
If $L$ is algebraic, the closure $\overline{L}$ is a holomorphic subvariety of $\overline{T^*M}$.
\end{lemma}
\begin{proof}
This is also clear from the definition.
\end{proof}

Here is the main technical theorem in this section.
\begin{theorem}
Let $L$ be an irreducible algebraic Lagrangian subvariety in $T^*M$. Then $L$ satisfies the following: 
\begin{enumerate}
    \item For any point $x$ of $p_{T^*M}(L)$, each connected component of $\pi^{-1}(x)\cap L\subset T^*M|_x$ is stable under the addition of $T^*_xN$.
    \item The quotient $L/T^*_NM\rightarrow p_{T^*M}(L)$ is a ramified covering.
\end{enumerate}
\end{theorem}
\begin{proof}
Since $L$ is algebraic, by the above lemma, $\overline{L}$ is algebraic in the fiber direction. Hence the restriction of the projection $\overline{T^*M}\rightarrow M$ to $\overline{L}$ is proper and algebraic. Then the projection image of $\overline{L}$ is an irreducible subvariety of $M$. We denote the smooth part of the projection image $p_{T^*M}({L})$ by $B$. 

Take a point $B$. We can take a local coordinate $x_1,..., x_n$ of the base space $M$ around $x$ such that $B$ is locally defined by $x_{m+1}=x_{m+2}=\cdots =x_n=0$. 

Take a smooth point $y\in L$ with $x=p_{T^*M}(y)$. Take a neighborhood $V$ of $x$ and take a connected component $U$ of $p_{T^*M}^{-1}(V)$, which is a neighborhood of $x$.

Consider the zero set $Z_i$ of $dx_i$ on $U$, which is a closed analytic subset of $U$. We set $U':=U\bs\lb \bigcup_i Z_i\rb$. 

Then the restriction of the symplectic form to $U'$ is
\begin{equation}
    dx_1\wedge d\xi_1+\cdots dx_m\wedge d\xi_m|_L=0.
\end{equation}
Since $dx_i$'s are independent, if one of $d\xi_i$ for $i=1,...,m$ is independent of $dx_i$'s, the above does not vanish. This implies that $dx_1,..., dx_m, d\xi_{m+1},..., d\xi_n$ are independent on $L$. In other words, $x_1,...,x_m, \xi_{m+1},..., \xi_n$ is a coordinate of $U'$. We also note that this implies that $\xi_1,..., \xi_m$ are fixed for a fixed $(x_1,..., x_m)$. A point in $U'$ with a fixed $x_1,..., x_m$ is in a plane defined by $(x_1,..., x_m, 0,,...,0, \xi_1,...., \xi_m)$. The holomorphy and the closedness of $L$ implies the first part of the theorem for smooth points.

For a singular point $y\in L$, there exists a sequence of smooth points $\{y_i\}_{i\in \bN}$ in $L$ which converges to $y$. Then the translation $\{y_i+\xi\}$ by $\xi \in T^*_NM$ is again in $L$ which converges to $y+\xi$. Since $L$ is closed $y+\xi\in L$.

Suppose $y$ is in the fiber of a singular point of $B$. By the irreducibility, there exists a sequence of points of the fibers in $B$ converging to $y$. Hence we can prove the assertion for $y$ similarly.

Since the action of $T^*_NM$ on $L$ is free, we obtain an $m$-dimensional variety as a quotient. Since the dimension of the projection image $p_{T^*M}(L)$ is the same as $L/T^*_NM$, we obtain a ramified covering map.
\end{proof}

Let $L$ be an algebraic Lagrangian. We first decompose $L$ into irreducible Lagrangians. We set $B_i:=p_{T^*M}(L_i)$ where $L_i$ is each irreducible component. 

\begin{definition}
Let $\cS$ be a subanalytic stratification $\cS$ of $M$. We say $\cS$ is adapted to $L$ if the following holds: For any $B_i$, there exists a subset $\cS_i$ of $\cS$ such that $\bigcup_{S\in \cS_i}S=B_i$.
\end{definition}

Let $L$ be an algebraic Lagrangian. Let $\cS$ be an adapted stratification. For each $S\in \cS$, we have a symplectic reduction $T^*M\rightarrow T^*S$. We set
\begin{equation}
    \widetilde{L_S}:=L_S+T^*_SM.
\end{equation}
By the above lemma, we have
\begin{lemma}
The reduction of $L$ into $T^*S$ is again Lagrangian for any $S$ and given by $\widetilde{L_S}/T^*_SM$.
\end{lemma}
Note that each $\widetilde{L_S}$ has again a property that the projection $L_S/T^*_SM\rightarrow S$ is a ramified covering. 
\footnotesize
\bibliographystyle{alpha}
\bibliography{fwkb.bib}

\noindent
Department of Mathematics, Graduate School of Science, Osaka University, Toyonaka, Osaka, 560-0043, tatsuki.kuwagaki.a.gmail.com
\end{document}